\documentclass{siamltex}
\usepackage[english]{babel}
\usepackage{graphicx,epstopdf,epsfig}
\usepackage{array,multirow,makecell}
\usepackage{amsfonts,epsfig,graphics,amsmath,amssymb}
\usepackage{subcaption}
\usepackage{booktabs} 
\usepackage{multicol}
\usepackage{url}
\usepackage{tikz}

\usepackage{ifthen}
\usepackage{float}
\floatstyle{ruled}
\newfloat{algorithm}{htb}{alg}
\floatname{algorithm}{Algorithm}

\newcounter{algo@row}
\newcounter{algo@rowindent}
\newcommand{\algofont}[1]{\textbf{#1}}
\newcommand{\algonumbersize}[1]{\scriptsize{#1}}
\newcommand{\algopreitem}[1][\arabic{algo@row}]{\texttt{\algonumbersize{#1}}}
\newcommand{\algoitemskip}{\hspace{\value{algo@rowindent}cc}}
\newenvironment{algo}{\vskip.3em\small%
  \begin{list}{\algopreitem\texttt{\algonumbersize{:}}}{%
      \usecounter{algo@row}%
      \setcounter{algo@rowindent}{0}%
      \setlength{\itemindent}{2em}%
      \setlength{\labelwidth}{2em}
      \setlength{\parsep}{0cm}%
    }%
}{
  \end{list}\vskip-.5em
}
\newcommand{\algonewnestedopen}[2]{
  \newcommand{#1}[1][]{%
    \ifthenelse{\equal{##1}{}}{\item}{\item[{\algopreitem[##1]}]}
    \algoitemskip\algofont{#2}%
    \addtocounter{algo@rowindent}{1}%
    \ignorespaces
  }
}
\newcommand{\algonewnestedaux}[2]{
  \newcommand{#1}[1][]{
    \addtocounter{algo@rowindent}{-1}
    \ifthenelse{\equal{##1}{}}{\item}{\item[{\algopreitem[##1]}]}
    \algoitemskip\algofont{#2}%
    \addtocounter{algo@rowindent}{+1}%
    \ignorespaces
  }
}
\newcommand{\algonewnestedclose}[2]{
  \newcommand{#1}[1][]{
    \addtocounter{algo@rowindent}{-1}
    \ifthenelse{\equal{##1}{}}{\item}{\item[{\algopreitem[##1]}]}
    \algoitemskip\algofont{#2}%
    \ignorespaces
  }
}
\newcommand{\algonewcommand}[2]{
  \newcommand{#1}[1][default]{
    \ifthenelse{\equal{##1}{default}}{\item}{\item[{\algopreitem[##1]}]}%
    \algoitemskip\algofont{#2}%
    \ignorespaces
  }%
}
\newcommand{\algonewkeyword}[2]{\newcommand{#1}{\algofont{#2}}}
\algonewcommand{\STATE}{\ignorespaces}
\algonewcommand{\INPUT}{Input: }
\algonewcommand{\pINPUT}{\phantom{Input: }}
\algonewcommand{\COMPUTE}{Compute: }
\algonewcommand{\OUTPUT}{Output: }
\algonewcommand{\pOUTPUT}{\phantom{Output: }}
\newcommand{\COMMENT}{\hfill // }
\algonewnestedopen{\IF}{if }
\algonewnestedaux{\ELSEIF}{else if }
\algonewnestedaux{\ELSE}{else }
\algonewnestedclose{\ENDIF}{end if }
\algonewnestedopen{\FOR}{for }
\algonewnestedclose{\ENDFOR}{end for }
\algonewnestedopen{\WHILE}{while }
\algonewnestedclose{\ENDWHILE}{end while }
\algonewcommand{\BREAK}{break}%
\algonewkeyword{\For}{for }%
\algonewkeyword{\To}{to }%
\algonewkeyword{\Do}{do }%
\algonewkeyword{\If}{if }%
\algonewkeyword{\REPEAT}{repeat }
\algonewkeyword{\Then}{then }%
\algonewkeyword{\Else}{else }%
\algonewkeyword{\End}{end }%
\algonewkeyword{\AND}{and }%
\algonewkeyword{\True}{true }%
\algonewkeyword{\False}{false }%
\algonewkeyword{\Call}{call }%
\algonewkeyword{\irbleigs}{irbleigs }%
\algonewkeyword{\tridiag}{tridiag}%
\algonewkeyword{\reorth}{reorth}%
\algonewkeyword{\Until}{until}

\usepackage{color}
\usepackage{multirow}
\usepackage{multicol}

\def\VV{\mathcal V}

\def\PP{\mathcal P}

\def\RR{\mathbb R}

\def\PP{\mathcal P}

\ifpdf
  \DeclareGraphicsExtensions{.eps,.pdf,.png,.jpg}
\else
  \DeclareGraphicsExtensions{.eps}
\fi

\title{An Accelerated Newton-GMRES Method for Multilinear PageRank\thanks{This manuscript is currently under peer review in an international journal.}}

\author{
Maryam Boubekraoui\thanks{Laboratory LAMAI, Faculty of Sciences and Technologies, Cadi
Ayyad University, Marrakech, Morocco. E-mail: maryam.boubekraoui@ced.uca.ma}
\and Ridwane Tahiri \thanks{Laboratory LAMAI, Faculty of Sciences and Technologies, Cadi
Ayyad University, Marrakech, Morocco. E-mail: tahiri0ridwane@gmail.com }
}

\usepackage{amsopn}

\begin{document}

\maketitle
\begin{abstract}
Modeling complex multiway relationships in large-scale networks is becoming more and more challenging in data science. The multilinear PageRank problem, arising naturally in the study of higher-order Markov chains, is a powerful framework for capturing such interactions, with applications in web ranking, recommendation systems, and social network analysis. It extends the classical Google PageRank model to a tensor-based formulation, leading to a nonlinear system that captures multi-way dependencies between states. Newton-based methods can achieve local quadratic convergence for this problem, but they require solving a large linear system at each iteration, which becomes too costly for large-scale applications. To address this challenge, we present an accelerated Newton-GMRES method that leverages Krylov subspace techniques to approximate the Newton step without explicitly forming the large Jacobian matrix. We further employ vector extrapolation methods,  including Minimal Polynomial Extrapolation (MPE), Reduced Rank Extrapolation (RRE), and Anderson Acceleration (AA), to improve the convergence rate and enhance numerical stability. Extensive experiments on synthetic and real-world data demonstrate that the proposed approach significantly outperforms classical Newton-based solvers in terms of efficiency, robustness, and scalability.
\end{abstract}



\section{Introduction}\label{sec1}
When a user writes a search query, the Google search engine captures a huge set of web pages that include words similar to those entered in search web pages with a set of the same words. These pages were important because Google created a system of scores that classified them called PageRank \cite{brin1998pagerank}. The PageRank algorithm computes a ranking for web pages based on the structure of the web. We look at the web as a huge directed graph whose nodes are all the web pages and whose edges are constituted by all the links between pages. Its primary goal is to determine the stationary distribution of a Markov chain that models the behavior of a random surfer. This web behavior can be represented by a key matrix called the Google matrix $P$. To define its elements, we consider $deg(i)$ as the number of pages that can be reached by a direct link from page $i$. Then, the elements of $P$ are defined as $p_{ij}=\frac{1}{deg(i)}$, if $deg(i) >0$ and there exists a link on the web from page $i$ to a certain page $j \neq i$.  In the case of rows $i$ where $deg(i) = 0$ we assume $p_{ij} = \frac{1}{n}$, where $n$ is the size of the
matrix, i.e., the cardinality of all the web pages. Let $\alpha$ be a probability less than one. With $\alpha$ a random surfer, it randomly transitions according to the column stochastic matrix $P$, and with probability $1-\alpha$ according to a fixed distribution given by the column stochastic vector $v$ \cite{gleich2015pagerank}.

\noindent The rank of each web page is represented in PageRank vector $\mathbf{x}$, reflecting its
position in the search results, which is the solution of the following linear system
 \begin{equation}
     \mathbf{x}=\alpha P \mathbf{x}+(1-\alpha) v.
 \end{equation}
To enhance ranking quality and address challenges in web networks, Gleich et al. expanded PageRank to higher orders by using the higher-order version of the Markov processes. This extension introduces memory into the process, meaning that the state at any given step relies on the previous $m$ states. This captures a memory of $m$ steps, unlike the standard Markov chain, which only takes the current state into account. Mathematically, an order-$m$ Markov chain $S$ is a stochastic process that satisfies
\begin{equation}
    \Pr(S_t = i_1 \mid S_{t-1} = i_2, \dots, S_1 = i_t) = \Pr(S_t = i_1 \mid S_{t-1} = i_2, \dots, S_{t-m} = i_{m+1}).
\end{equation}
Higher-order PageRank model is introduced by modeling a random surfer using a higher-order Markov chain. The surfer transitions according to this higher-order chain with a probability of $\alpha$, while with a probability of $1-\alpha$, the surfer teleports based on the distribution $v$. The limiting probability distribution of this chain is the higher-order PageRank, as explained in \cite{gleich2015multilinear, boubekraoui2023vector} and the references therein. However, for very large web graphs, this approach becomes impractical due to the memory needed to represent the solution.  In consequence, an approximation for the higher-order PageRank, referred to as a multilinear PageRank vector, is introduced in \cite{gleich2015multilinear}.  The multilinear PageRank problem consists in finding a stochastic solution $\mathbf{x} \in \RR^{n}$ to the following equation
\begin{equation}\label{mpr}
    \mathbf{x}=\alpha \PP \mathbf{x}^{m-1}+ (1-\alpha) v,
\end{equation}
where $\PP \mathbf{x}^{m-1}$ is a vector in $\RR^{n}$, whose $i$-th component is
\begin{equation}
   \sum \limits_{i_2,\dots,i_{m}=1}^{n} p_{i i_{2}\dots i_{m}} \mathbf{x}_{i_2}  \dots \mathbf{x}_{i_m},
\end{equation}
With \( \PP = (p_{i i_{2} \dots i_{m}}) \) being an order-\( m \) transition probability tensor that represents a \((m-1)\)th order Markov chain modeling the higher-order interactions between web pages, we naturally have the properties  
\[
p_{i i_2 \dots i_m} \geq 0, \quad \sum_{i=1}^{n} p_{i i_2 \dots i_m} = 1, \quad \text{for all } (i_2 \dots  i_m).
\]
It is possible to express the multilinear PageRank problem (\ref{mpr}) in its matricized form as
\begin{equation}\label{equamuti}
    \mathbf{x}=\alpha R(\underbrace{\mathbf{x} \otimes \mathbf{x} \dots  \otimes \mathbf{x}}_{m-1})+(1-\alpha)v,
\end{equation}
where $R$ is the $n$-by-$n^{m-1}$ fattening of $\PP$ along the first index (see \cite{kolda2009tensor}), and $\otimes$ is the Kronecker product.  Alternatively, the multilinear PageRank problem can also be formulated as a Z-eigenvalue problem \cite{qi2018tensor}, represented by the equation
\begin{equation}\label{multipr} \mathbf{x}=\mathcal{M} \mathbf{x}^{m-1},\quad \text{with}\;\; \mathcal{M}= \alpha \mathcal{P}+ (1-\alpha ) \mathcal{V},  \end{equation}
where $\VV=v \otimes e \otimes \dots \otimes e$ with $\VV_{i_{1},i_{2},..,i_{m}}=v_{i_{1}}$.

\noindent  For tackling the multilinear PageRank equation (\ref{mpr}) or (\ref{equamuti}), numerous algorithms have been developed, with early contributions including the work of Gleich et al.~\cite{gleich2015multilinear}, Meini and Poloni~\cite{meini2018perron}, and Cipolla et al.~\cite{cipolla2020extrapolation}. In their seminal study, Gleich et al.~\cite{gleich2015multilinear} introduced five distinct iterative approaches for this high-dimensional problem, paving the way for subsequent developments. At the heart of these methods lies one of the simplest and most intuitive strategies: the fixed-point iteration, which evolves through the recurrence relation

\begin{equation}
    g(\mathbf{x}_{k}) = \mathbf{x}_{k+1}, \quad k = 0, 1, 2, \dots
\end{equation}
where the function $g(\mathbf{x})$ is defined as
\begin{align}
g: \; \Omega &\longrightarrow \Omega \nonumber \\
\mathbf{x} &\longrightarrow g(\mathbf{x}) = \alpha R(\underbrace{\mathbf{x} \otimes \mathbf{x} \dots \otimes \mathbf{x}}_{m-1}) + (1-\alpha)v, \label{non-equ}
\end{align}
with $\Omega$ represents the probability set given as

\begin{equation}
\Omega = \left\{ \mathbf{x} \in \mathbb{R}^{n}, \; \sum_{i=1}^{n} \mathbf{x}_{i} = 1, \; \mathbf{x}_{i} \geq 0, \; i = 1, \dots, n \right\} \subset \mathbb{R}^{n}.
\end{equation}

 \noindent This iteration scheme mirrors the well-known Higher-Order Power Method (HOPM), a classical approach in tensor computations. Several variants have been proposed to improve its robustness and convergence properties. Notably, the shifted symmetric higher-order power method (SS-HOPM) introduces a shift parameter to stabilize the iteration and enlarge the convergence domain, particularly when dealing with tensors that are not strictly positive~\cite{gleich2015multilinear,kolda2011shifted}. Despite these improvements, fixed-point–type schemes may still converge slowly for large-scale or ill-conditioned multilinear PageRank problems. To address this limitation, acceleration techniques based on vector extrapolation have been successfully employed. For example, Cipolla, Redivo-Zaglia, and Tudisco~\cite{cipolla2020extrapolation} applied the simplified topological $\epsilon$-algorithm in its restarted form to enhance the convergence of fixed-point and SS-HOPM schemes. More recent studies~\cite{bentbib2024extrapolation,boubekraoui2023vector} applied vector extrapolation methods such as Aitken’s $\Delta^2$ process, the minimal polynomial extrapolation (MPE) and  the reduced rank extrapolation (RRE) in a tailored and effective manner, achieving substantial reductions in iteration counts and computational time for fixed-point schemes. Additionally, Lai et al.~\cite{lai2023anderson} introduced the Periodical Anderson Acceleration for Relaxed Fixed-Point Iteration (PAA-RFP), offering another powerful tool for improving convergence rates.

\noindent Beyond fixed-point iteration, another powerful approach to solving equation (\ref{mpr}) is Newton’s method. By leveraging the Newton-Raphson technique, this method iteratively refines the solution to the nonlinear system $f(\mathbf{x}) = g(\mathbf{x}) - \mathbf{x}=0$. For a given initial guess $\mathbf{x}_0 \in \mathbb{R}^n$, the Newton sequence for solving equation (\ref{equamuti}) follows the iteration scheme:
\begin{equation}\label{eqnewton}
 \begin{cases}
\begin{aligned}
&\text{For } k = 0, 1, \dots \text{ until convergence, do:} \\
&\quad \text{Solve} \quad J_f\left(\mathbf{x}_k\right) \delta_k = -f\left(\mathbf{x}_k\right), \\
&\quad \text{and set} \quad \mathbf{x}_{k+1} = \mathbf{x}_k + \delta_k.
\end{aligned}
\end{cases}   
\end{equation}
where $J_{f}\left(\mathbf{x}_{k}\right)$ is the Jacobian matrix of $f$, defined as
\begin{equation}
 J_{f}\left(\mathbf{x}_{k}\right)=\alpha {R}({I} \otimes \mathbf{x}_{k} \otimes \cdots \otimes \mathbf{x}_{k}+\mathbf{x}_{k} \otimes I \otimes\mathbf{x}_{k} \otimes \cdots \otimes \mathbf{x}_{k}+ \dots +\mathbf{x}_{k} \otimes \cdots \otimes \mathbf{x}_{k} \otimes I)-I.  
\end{equation}
Unlike fixed-point iteration, Newton’s method exhibits a quadratic convergence rate when initialized sufficiently close to the solution. However, its efficiency comes at the cost of solving a linear system at each step, which can become computationally expensive, particularly for large-scale problems. Each iteration of Newton’s method typically involves solving a Newton equation, often requiring matrix inversion or using expensive direct methods, which can quickly become prohibitive as the problem size grows. Another issue with Newton’s method for computing the multilinear PageRank vector is its convergence behavior. It has been shown by Gleich et al. \cite{gleich2015multilinear} that for a third-order tensor, Newton's iteration converges quadratically only when the damping factor $\alpha$ is less than $1/2$. When this condition is not met, convergence may become slower or even fail, posing a significant challenge for certain applications.  

\noindent To address this issue, Guo et al. \cite{guo2018modified} proposed a modified Newton method for computing the multilinear PageRank vector. They showed that for a third-order tensor, when $\alpha < 1/2$, and with a suitable initial guess, the sequence of iterative vectors generated by the modified Newton method is monotonically increasing and converges to the solution of the equation (\ref{mpr}). While this approach provides improvements in certain cases, the challenge of matrix inversion remains, and the method's behavior when $\alpha \geq 1/2$ requires further consideration.

\noindent To overcome these challenges and enhance the efficiency and scalability of Newton’s method, this work proposes the use of accelerated nonlinear Krylov methods. These methods aim to approximate the solution of the Newton equation without the need for direct matrix inversion, thereby reducing the computational burden. Additionally, we integrate vector extrapolation techniques to further accelerate convergence, ensuring improved accuracy and performance, even when the damping factor $\alpha$ is not ideal for Newton’s method.

\noindent The structure of this work is as follows. Section \ref{sec1} introduces the key notations and definitions. Section \ref{sec2} presents the classical Newton method for solving the multilinear PageRank problem, discussing its convergence and challenges. Section \ref{sec3} introduces Krylov-based solvers, specifically GMRES, and their non-linear versions combined with Newton iterations. In section \ref{sec5}, we explain how Krylov solvers are enhanced with extrapolation methods, such as Minimal Polynomial Extrapolation (MPE), Reduced Rank Extrapolation (RRE), and Anderson acceleration, to speed up convergence.  Finally, Section \ref{sec5k} presents numerical experiments comparing the proposed methods with classical solvers, demonstrating their effectiveness for large-scale problems.

\section{Preliminaries}\label{sec1k}
To ensure clarity in the subsequent section, we briefly introduce some fundamental terminology.

Let $n$ be a positive integer, and denote by $\langle n \rangle$ the index set $\{1, 2, \dots, n\}$. A tensor $\mathcal{A}$ of order $m$ and dimensions $n_1 \times \cdots \times n_m$ over the real field $\mathbb{R}$ is defined as
\[
\mathcal{A} = (a_{i_1 \cdots i_m}), \quad \text{where } a_{i_1 \cdots i_m} \in \mathbb{R}, \text{ and } i_j \in \langle n_j \rangle \text{ for } j = 1, \dots, m.
\]
When all dimensions are equal, i.e., $n_1 = \cdots = n_m = n$, the tensor is referred to as an order-$m$, dimension-$n$ tensor. The set of all such tensors is denoted by $\mathbb{R}^{[m,n]}$. In particular, when $m = 2$, the notation $\mathbb{R}^{[2,n]}$ corresponds to the set of all $n \times n$ real matrices, and when $m = 1$, $\mathbb{R}^{[1,n]}$ simplifies to $\mathbb{R}^n$, the space of all $n$-dimensional real vectors. 

A vector \( z \in \mathbb{R}^n \) is said to be stochastic if its entries are non-negative and sum to one.
\begin{definition}
    Let \( z \in \mathbb{R}^n \) be a stochastic vector and the zero vector \( 0 \in \mathbb{R}^n \). The projection of \( z \) is given by  
\begin{equation}
    \operatorname{proj}(z) = \frac{z^{+}}{\|z^{+}\|_1},
\end{equation}
where  $ z^{+} = \max(z, 0)$ is the non-negative part of $z$, and with  \( \operatorname{proj}(z) \) is also a stochastic vector.
\end{definition}

\noindent Matrix theory provides powerful tools for studying stability and convergence in numerical methods. Two important classes of matrices in this context are Z-matrices and M-matrices.
\begin{definition}[\cite{berman1994nonnegative}] \label{M-matrix}
    A matrix \( A \) is called a Z-matrix if all its off-diagonal entries are non-positive. Formally, it can be expressed as
    \[
    A = s I - B,
    \]
    where \( B \) is a non-negative matrix and \( s > 0 \). Furthermore, \( A \) is an \( M \)-matrix if \( s \geq \rho(B) \), and a nonsingular \( M \)-matrix if \( s > \rho(B) \), where \( \rho(B) \) denotes the spectral radius of \( B \).
\end{definition}

\noindent The following theorem establishes a key equivalence property of nonsingular M-matrices, which will play a crucial role in the subsequent developments of this paper
\begin{theorem}[\cite{berman1994nonnegative, fiedler1962matrices}] \label{thm:Mmatrix}  
For a Z-matrix \( A \), it is a nonsingular  M-matrix if and only if its inverse \( A^{-1} \) has all entries non-negative, i.e., \( A^{-1} \geq 0 \).  
\end{theorem}

\section{Newton's Method for Multilinear PageRank}\label{sec2}
In this section, we present the Newton iteration for computing the multilinear PageRank problem. Furthermore, we discuss the sufficient condition for its convergence.

\noindent At the core of this approach is a Newton-like iteration scheme designed to solve the nonlinear system
\begin{equation}
    f(\mathbf{x})=\alpha R(\underbrace{\mathbf{x} \otimes \mathbf{x} \dots \otimes \mathbf{x}}_{m-1}) + (1-\alpha)v-\mathbf{x}=0,
\end{equation}
where $f$ is non-linear function defined as $f(\mathbf{x})=g(\mathbf{x})-\mathbf{x}$  as discussed in the introduction.  Our goal is to iteratively find the solution to this system, which corresponds to the multilinear PageRank vector $\mathbf{x} \in \Omega$. Given an initial guess $\mathbf{x}_0$, the next iterate is computed by:
\begin{equation}
    \mathbf{x}_{k+1} = \mathbf{x}_k - J_{f}^{-1}(\mathbf{x}_k) f(\mathbf{x}_k),
\end{equation}
where $J_{f}(\mathbf{x}_k)$ is the Jacobian matrix of $f(\mathbf{x})$ evaluated at the current iterate $\mathbf{x}_k$, and $J_{f}^{-1}(\mathbf{x}_k)$ is its inverse. The iteration continues until the residual $\| f(\mathbf{x}_k) \|$ is sufficiently small, indicating convergence to the solution. Through this process, the non-singularity of the Jacobian matrix becomes crucial for the convergence of the method. If the Jacobian is singular or nearly singular at any point, the method may fail to converge or produce inaccurate solutions. In the following, we present the necessary conditions under which the Jacobian matrix remains non-singular, thus ensuring the convergence of the Newton method.
\begin{proposition}\label{non-singular}
     Let $\PP \in \RR^{[m,n]}$ be a transition probability tensor and $\mathbf{x}$ be a non-negative vector such that $e^{T} \mathbf{x} \leq 1$. Let $R$ be the flatting of $\PP$ along the first index. Then the  following matrix 
  \begin{equation}\label{Jac}
      -J\left(\mathbf{x}\right)=I-\alpha {R}({I} \otimes \mathbf{x} \otimes \cdots \otimes \mathbf{x}+\mathbf{x} \otimes I \otimes\mathbf{x} \otimes \cdots \otimes \mathbf{x}+\dots +\mathbf{x} \otimes \cdots \otimes \mathbf{x} \otimes I)
  \end{equation}
  is a non-singular matrix when $ \alpha  
     <  \dfrac{1}{(m-1)}$.
\end{proposition}
\begin{proof}
The equation (\ref{Jac}), can be written as
    \begin{equation}
        -J\left(\mathbf{x}\right)= I-B,
    \end{equation}
    where $B= \alpha {R}({I} \otimes \mathbf{x} \otimes \cdots \otimes \mathbf{x}+\mathbf{x} \otimes I \otimes\mathbf{x} \otimes \cdots \otimes \mathbf{x}+\dots +\mathbf{x} \otimes \cdots \otimes \mathbf{x} \otimes I)$. This structure suggests that  the matrix $-J(\mathbf{x})$ might be a Z-matrix (see Definition \ref{M-matrix} with $s=1$) because $B$ is a non-negative matrix since R is a column stochastic matrix (since R is a
stochastic tensor), and $\mathbf{x}$  is a non-negative vector. However, to prove that $-J(x)$ is non-singular, we need to show that $\rho(B) <1$ (see Definition \ref{M-matrix}). To do this, we use the spectral radius inequality
    $$\rho(B)\leq  \Vert B \Vert_{1}.$$
    Therefore, we have 
    \begin{equation}\label{eq77}
       \rho(B) \leq \alpha  \Vert R \Vert_{1} \Big(\Vert {I} \otimes \mathbf{x} \otimes \cdots \otimes \mathbf{x} \Vert_{1} + \Vert \mathbf{x} \otimes I \otimes\mathbf{x} \otimes \cdots \otimes \mathbf{x} \Vert_{1}+\dots + \Vert \mathbf{x} \otimes \cdots \otimes \mathbf{x} \otimes I \Vert_{1} \Big). 
    \end{equation}
    Next, using the property of the norm of a Kronecker product,
$\|\mathbf{a} \otimes \mathbf{b}\| = \|\mathbf{a}\| \|\mathbf{b}\|$,  and the fact that $ \Vert I \Vert_{1}=1$, we get
    $$\Vert {I} \otimes \mathbf{x} \otimes \cdots \otimes \mathbf{x} \Vert_{1}=\Vert \mathbf{x} \otimes I \otimes\mathbf{x} \otimes \cdots \otimes \mathbf{x} \Vert_{1}=\dots = \Vert \mathbf{x} \otimes \cdots \otimes \mathbf{x} \otimes I \Vert_{1}= \Vert \mathbf{x} \Vert_{1}^{m-2}=(e^{T} \mathbf{x})^{m-2}.$$
By applying the previous equation, along with the property that $\Vert R \Vert_{1}=1 $ and $e^{T} \mathbf{x} \leq 1$, we can express equation (\ref{eq77}) as
    $$\rho(B) \leq \alpha (m-1) (e^{T} \mathbf{x})^{m-2} \leq \alpha (m-1). $$
    Since $\alpha  < \frac{1}{m-1}$, then $\rho(B)<1$ which follow that $-J\left(\mathbf{x}\right)$ is non-singular.
\end{proof}

\subsection{Convergence Analysis of the Newton Iteration}
We now study the convergence of the Newton method applied to the multilinear PageRank problem. A first convergence result was established in \cite{gleich2015multilinear} for tensors of order three. Here, we extend this result to tensors of any order \(m\) and provide a simpler and more general proof. The convergence is stated in the following theorem.

\begin{theorem} \label{newconv}
  Let \( \alpha < \frac{1}{m - 1} \), and consider the initial vector \( \mathbf{x}_0 = 0 \). Then, the Newton iteration
  \begin{equation} \label{newiter}
    \mathbf{x}_{k+1} = \mathbf{x}_k - J_{f}(\mathbf{x}_k)^{-1} f(\mathbf{x}_k), \quad k = 0, 1, 2, \dots
  \end{equation}
  converges to the unique solution \( \mathbf{x} \) of the multilinear PageRank equation.
\end{theorem}
\begin{proof}
We begin by proving that the Newton iteration defined in~\eqref{newiter} is well-defined and produces feasible iterates. Specifically, we show that
\begin{equation}\label{eq79}
   f(\mathbf{x}_k) \geq 0, \quad \text{for all } \mathbf{x}_k \geq 0 \text{ with } \mathbf{z}_k := e^T \mathbf{x}_k \leq 1,
\end{equation}
holds for all iterations \( k \). The proof proceeds by induction.

At iteration \( k = 0 \), we initialize with \( \mathbf{x}_0 = 0 \), which implies \( \mathbf{z}_0 = 0 \leq 1 \), and
\[
f(\mathbf{x}_0) = (1 - \alpha)\mathbf{v} \geq 0,
\]
so the condition is satisfied initially.

Assuming now that~\eqref{eq79} holds for some index \( k \), we aim to prove that it remains true at the next step. Let \( \delta_k = \mathbf{x}_{k+1} - \mathbf{x}_k \). Expanding \( f \) at \( \mathbf{x}_k \) via Taylor series gives:
\begin{equation}\label{taylor-detailed}
f(\mathbf{x}_k + \delta_k) = f(\mathbf{x}_k) + J_f(\mathbf{x}_k) \delta_k + \frac{1}{2} H_f(\mathbf{x}_k)(\delta_k \otimes \delta_k) + \cdots + \frac{1}{(m-1)!} f^{(m-1)}(\mathbf{x}_k)(\delta_k \otimes \dots \otimes \delta_k),
\end{equation}
where \( H_f(\mathbf{x}_k) \) is the Hessian of \( f \) at \( \mathbf{x}_k \). 
Since the Newton step satisfies \( J_f(\mathbf{x}_k)\delta_k = -f(\mathbf{x}_k) \), the linear term cancels and we get:
\[
f(\mathbf{x}_{k+1}) = \frac{1}{2} H_f(\mathbf{x}_k)(\delta_k \otimes \delta_k) + \cdots + \frac{1}{(m-1)!} f^{(m-1)}(\mathbf{x}_k)(\delta_k \otimes \dots \otimes \delta_k).
\]
We now verify that each term in this sum is non-negative. Consider first the term:
\begin{equation}\label{first-term}
\frac{1}{2} H_{f}(\mathbf{x}_{k}) (\delta_{k} \otimes \delta_{k}).
\end{equation}
For any \( m \geq 3 \), the Hessian of \( f \) is given by:
\begin{equation}\label{hessian-general}
H_{f}(\mathbf{x}) = \alpha R \left( I \otimes I \otimes \mathbf{x} \otimes \cdots \otimes \mathbf{x} + \cdots + \mathbf{x} \otimes I \otimes I \otimes \cdots \otimes \mathbf{x} + \cdots + \mathbf{x} \otimes \cdots \otimes I \otimes I \right),
\end{equation}
where each term is a Kronecker product involving two identity matrices and \( m - 3 \) entries of \( \mathbf{x} \). Since all entries of \( R \) are nonnegative and given that \( \mathbf{x}_k \geq 0 \), the entire Hessian is entrywise non-negative.

Moreover, from the hypothesis and Proposition~\ref{non-singular}, \( -J_f(\mathbf{x}_k) \) is a non-singular \( M \)-matrix, so its inverse is non-negative. Since \( f(\mathbf{x}_k) \geq 0 \), this implies \( \delta_k = -J_f(\mathbf{x}_k)^{-1} f(\mathbf{x}_k) \geq 0 \). Thus,
\[
\frac{1}{2} H_{f}(\mathbf{x}_{k}) (\delta_{k} \otimes \delta_{k}) \geq 0.
\]

This argument extends similarly to higher derivatives. For example, the third derivative has the form
\[
f^{(3)}(\mathbf{x}) = \alpha R \left( I \otimes I \otimes I \otimes \cdots \otimes \mathbf{x} + \cdots + \mathbf{x} \otimes I \otimes I \otimes I \otimes \cdots \otimes \mathbf{x} + \cdots + \mathbf{x} \otimes \cdots \otimes I \otimes I \otimes I \right),
\]
where the terms involve replacing positions of \( \mathbf{x} \) in the Kronecker product with identity matrices \( I \), analogous to the procedure in the Hessian calculation. For the same reason as in the previous case, we conclude that
\[
f^{(3)}(\mathbf{x}_k)(\delta_k \otimes \delta_k \otimes \delta_k) \geq 0.
\]

This reasoning holds for higher derivatives as well, up to the \( (m-1) \)-th derivative, which is given by:
\begin{equation}
f^{(m-1)}(\mathbf{x}) = \alpha R (I \otimes I \otimes I \otimes \cdots \otimes I + \cdots + I \otimes I \otimes I \otimes I \otimes \cdots \otimes I + \cdots + I \otimes \cdots \otimes I \otimes I \otimes I). \tag{3.12}
\end{equation}
For the simplest case where \( m = 3 \), we have
\begin{equation}
f^{(m-1)}(\mathbf{x})=H_{f}(\mathbf{x}) = \alpha (I \otimes I + I \otimes I) = 2\alpha (I \otimes I).
\end{equation}
In this case, the expression for \( f(\mathbf{x}_{k+1}) \) simplifies to:
\begin{equation}\label{simplecase}
f(\mathbf{x}_{k+1}) = \alpha R (\delta_k \otimes \delta_k),
\end{equation}
as shown in Gleich et al.~\cite{gleich2015multilinear}.   Therefore, for all \( m \geq 3 \), we conclude:
\[
f(\mathbf{x}_{k+1}) \geq 0, \quad \text{for all}\;\; \mathbf{x}_{k+1} \geq 0.
\]
\noindent It remains to prove that the next iterate remains within the feasible set. For that,  taking summations on both sides of the equation
$$ -J_{f}(\mathbf{x}_{k}) \delta_{k}=f(\mathbf{x}_{k}),$$
we arrive at the next result
$$ -e^{T} J_{f}(\mathbf{x}_{k})(\mathbf{x}_{k+1}-\mathbf{x}_{k}) = e^{T} f(\mathbf{x}_{k}),$$
which implies
$$ (1-\alpha(m-1) \mathbf{z}_{k}^{m-2})(\mathbf{z}_{k+1}-\mathbf{z}_{k})=\alpha  \mathbf{z}_{k}^{m-1}+(1-\alpha)-\mathbf{z}_{k}.$$
Therefore
\begin{align*}
  \mathbf{z}_{k+1}&=  \dfrac{\alpha \mathbf{z}_{k}^{m-1}+(1-\alpha)-\mathbf{z}_{k} }{1-\alpha (m-1) \mathbf{z}_{k}^{m-2}}+\mathbf{z}_{k}
  \\&=\dfrac{(1-(m-1))\alpha  \mathbf{z}_{k}^{m-1} + (1-\alpha)-\mathbf{z}_{k}}{1-\alpha (m-1) \mathbf{z}_{k}^{m-2}}.\\
\end{align*}
We aim to prove that \( \mathbf{z}_{k+1} \leq 1 \) if \( \mathbf{z}_{k} \leq 1 \). We start by  the denominator, since \( \mathbf{z}_{k} \leq 1 \), we have \( \mathbf{z}_{k}^{m-2} \leq 1 \), which implies:
$$
1 - \alpha (m-1) \mathbf{z}_{k}^{m-2} \geq 1 - \alpha (m-1).
$$
To ensure that the denominator remains positive, we require that \( 1 - \alpha (m-1) > 0 \), which simplifies in the condition $\alpha < \frac{1}{m-1}$. In the other hand, the numerator can be simplified as
$$(2 - m)\alpha \mathbf{z}_{k}^{m-1} + (1 - \alpha) - \mathbf{z}_{k}.$$
Since  $\mathbf{z}_{k} \leq 1$, we know that $ \mathbf{z}_{k}^{m-1} \leq 1$. Thus, we further bound this expression
$$
(2 - m)\alpha \mathbf{z}_{k}^{m-1}+(1 - \alpha) - \mathbf{z}_{k} \leq (2 - m)\alpha+(1 - \alpha) -1=(1-m) \alpha.
$$
 In the worst-case scenario where \( \mathbf{z}_{k} = 1 \), the numerator to be non-positive, we need \( (1 - m)\alpha \leq 0 \), which holds since \( m > 1 \) and \( \alpha \geq 0 \). Thus, under the condition \( \alpha < \frac{1}{m-1} \), the numerator is bounded, ensuring that \( \mathbf{z}_{k+1} \leq 1 \).
This completes all the inductive conditions for (\ref{eq79}). 

\noindent Now, we prove that the sequence of iterates $\{ \mathbf{x}_{k}\}_{k=0}^{\infty}$ converges. Firstly, we prove that $\{ \mathbf{x}_{k}\}$ is an increasing sequence. For that, we consider the equation $\delta_{k}=-J_{f}(\mathbf{x}_{k})^{-1} f(\mathbf{x}_{k})$.  Since the iteration $\mathbf{x}_{k}$
  is well-defined and $-J_{f}(\mathbf{x}_{k})$ is a non-singular M-matrix, its inverse $-J_{f}(\mathbf{x}_{k})^{-1}$  is non-negative (see  Theorem \ref{thm:Mmatrix}). This implies that  $\delta_{k} \geq 0$, and consequently $\mathbf{x}_{k+1} \geq \mathbf{x}_{k}$. Secondly, we will need to show that the sequence $\{ \mathbf{x}_{k}\}$  is bounded.  We proceed by induction, when $k=0$ it is clear we have $\mathbf{x}_{0}=0 \leq \mathbf{y}$, for any non-negative vector $\mathbf{y}$. We now assume $\mathbf{x}_{k} \leq \mathbf{y}$, and consider the result for $k+1$. We have
  \begin{equation}
      \mathbf{y}-\mathbf{x}_{k+1}=\mathbf{y}-\mathbf{x}_{k}+J(\mathbf{x}_{k})^{-1} f(\mathbf{x}_{k}) \geq 0,
  \end{equation}
 which implies that  $\mathbf{x}_{k+1} \leq \mathbf{y}$. Consequently, the sequences $\{ \mathbf{x}_{k}\}_{k=0}^{\infty}$ converge to the unique solution with $\mathbf{x} \geq 0$ and $e^{T} \mathbf{x} \leq 1$. 
\end{proof}

Having shown that the Newton iteration converges to the unique solution of the multilinear PageRank problem, we now proceed to determine the rate at which the iterates approach this solution. This is described in the following theorem.
\begin{theorem}
 Under the same condition, the Newton iteration defined in~\eqref{newiter} converges quadratically to the unique solution \( \mathbf{x} \) of the multilinear PageRank problem.
\end{theorem}
\begin{proof}
To analyze the convergence rate of the Newton method, we consider the Taylor expansion of $f(\mathbf{x}_{k+1})$ around $\mathbf{x}_{k}$ as 
$$f(\mathbf{x}_{k+1})=f(x_{k})+J_{f}(\mathbf{x}_{k})(\mathbf{x}_{k+1}-\mathbf{x}_{k})+O(\mathbf{x}_{k+1}-\mathbf{x}_{k}),$$
where the remainder term satisfies
\begin{equation}\label{remainder}
  \Vert O(\mathbf{x}_{k+1}-\mathbf{x}_{k})  \Vert \leq M \Vert \mathbf{x}_{k+1}-\mathbf{x}_{k} \Vert^{2},  
\end{equation}
with $M>0$, using the Newton iteration, we obtain
$$ J_{f}(\mathbf{x}_{k})(\mathbf{x}_{k+1}-\mathbf{x}_{k})+f(\mathbf{x}_{k})=0,$$
which implies
\begin{equation}
    f(\mathbf{x}_{k+1})=O(\mathbf{x}_{k+1}-\mathbf{x}_{k}).
\end{equation}
Next, upon multiplying by $e^{T}$, leads to
\begin{equation}\label{eq715}
   e^{T} f(\mathbf{x}_{k+1})=O(\mathbf{z}_{k+1}-\mathbf{z}_{k}).
\end{equation}
On the other hand, the equation
$$ -J_{f}(\mathbf{x}_{k}) \delta_{k}=f(\mathbf{x}_{k}),$$
also holds, whereby multiplying  by $e^{T}$, we find
\begin{equation}
    e^{T} f(\mathbf{x}_{k})=-e^{T} J_{f}(\mathbf{x}_{k})(\mathbf{x}_{k+1}-\mathbf{x}_{k})=(1-\alpha(m-1) \mathbf{z}_{k}^{m-2})(\mathbf{z}_{k+1}-\mathbf{z}_{k}). 
\end{equation}
Using a similar argument for $k+1$, we obtain
\begin{equation}\label{eq717}
    e^{T} f(\mathbf{x}_{k+1})=-e^{T} J_{f}(\mathbf{x}_{k+1})(\mathbf{x}_{k+2}-\mathbf{x}_{k+1})=(1-\alpha(m-1) \mathbf{z}_{k+1}^{m-2})(\mathbf{z}_{k+2}-\mathbf{z}_{k+1}). 
\end{equation}
Substituting equation (\ref{eq717}) into (\ref{eq715}) gives
$$
  (1-\alpha(m-1) \mathbf{z}_{k+1}^{m-2})(\mathbf{z}_{k+2}-\mathbf{z}_{k+1})=O(\mathbf{z}_{k+1}-\mathbf{z}_{k}).
$$
Thus
$$
    (\mathbf{z}_{k+2}-\mathbf{z}_{k+1})=\dfrac{1}{(1-\alpha(m-1) \mathbf{z}_{k+1}^{m-2})}O(\mathbf{z}_{k+1}-\mathbf{z}_{k}).
$$
Since the sequence converges and $\mathbf{z}_{k} \leq 1 $ for all sufficiently large $k $, we have
\begin{equation}
 (\mathbf{z}_{k+2}-\mathbf{z}_{k+1}) \leq \dfrac{1}{1-\alpha(m-1)}O(\mathbf{z}_{k+1}-\mathbf{z}_{k}).
\end{equation}
 By leveraging this and (\ref{remainder}) we find that there exists a constant $C>0$  such that
 $$ \Vert \mathbf{z}_{k+2}-\mathbf{z}_{k+1} \Vert \leq C \Vert \mathbf{z}_{k+1}-\mathbf{z}_{k} \Vert^{2}, $$
where $C=\dfrac{M}{1-\alpha(m-1)}>0$.  This is positive because, $M>0$ and $1-\alpha(m-1)>0$, as $\alpha <\dfrac{1}{m-1}$.
\end{proof}

As noted in Theorem~\ref{newconv}, the Newton iteration starting with $\mathbf{x}_{0} = 0$ behaves differently depending on the value of $\alpha$. When $\alpha < \frac{1}{m-1}$, it gradually grows the solution until it becomes stochastic and converges to the correct PageRank vector. However, for problems when $\alpha >\frac{1}{m-1}$, this iteration often converges into a non-stochastic solution. In order to tackle this issue and to make the Newton algorithm practical for such cases, Gleich et al. \cite{gleich2015multilinear} placed a normalization step after each Newton iteration that ensures a stochastic solution. The following algorithm includes this modification, which ensures convergence to a stochastic solution during Newton iterations for multilinear PageRank computations.
\begin{algorithm}[H]
    \caption{ Newton Method for Multilinear PageRank }
    \label{alg:inexact_newton_projection}
    \begin{algo}
    \INPUT{Given a transition probability tensor $\PP$, $\alpha \in [0,1)$, maximum number of iterations $K_{max}$, and a termination tolerance $\epsilon$}
       \OUTPUT{The multilinear PageRank vector $\mathbf{x}$.}
        \STATE Initialize $k = 0$ and set $\mathbf{x}_0 = v$ as the initial guess,
      \STATE  \REPEAT
            \STATE Solve $J_{f}\left(\mathbf{x}_{k}\right) \delta_k = -f\left(\mathbf{x}_{k}\right)$,
            \STATE Update $\mathbf{x}_{k+1} = \text{proj}(\mathbf{x}_{k} + \delta_k)$,
\STATE $k=k+1$,
           \STATE \Until{$\|f(\mathbf{x}_{k})\|_{1} < \epsilon$}.
    \end{algo}
\end{algorithm}
While the Newton method offers superior convergence behavior, as we have seen that the sequence generated by Newton's iteration will converge quadratically to the solution of the nonlinear equation $f(\mathbf{x}) = 0$, there is a disadvantage with the Newton method. At every Newton iteration step, we need to solve the linear system $J_{f}\left(\mathbf{x}_{k}\right) \delta_k = -f\left(\mathbf{x}_{k}\right)$, which can be computationally expensive, especially for high dimensions. To address this, we propose using Krylov subspace methods to solve the linear system approximately. This approach, which combines Newton's method with Krylov subspace methods, is known as nonlinear Krylov subspace projection methods \cite{brown1990hybrid}.  In the next section, we will present how this method can be specifically applied to tackle the multilinear PageRank problem.

\section{Non-linear Krylov method for multilinear PageRank}\label{sec3}
 In this section, we provide an overview of Krylov subspace methods, with a focus on the GMRES algorithm, before discussing its nonlinear extensions and finite-difference variants in the context of multilinear PageRank.
\subsection{Krylov subspace methods}
Solving large linear systems of the form \( A x = b \), where \( A \in \mathbb{C}^{n \times n} \) and \( b \in \mathbb{C}^{n} \), is fundamental in scientific computing, arising in various applications. However, their solution can be computationally expensive, especially in time-dependent, nonlinear, or inverse problems requiring multiple solves. Efficient solvers are therefore essential for improving numerical performance. Direct methods provide exact solutions via matrix factorization \cite{duff2017direct} but become impractical for large problems due to high storage and computational costs. Iterative methods offer a more efficient alternative, computing approximate solutions through matrix-vector multiplications without explicit factorizations or full matrix storage \cite{bellavia2013matrix, knyazev2001toward}.

Among iterative methods, Krylov subspace methods are particularly effective for solving large linear systems. These methods iteratively build an approximate solution within a subspace spanned by successive matrix-vector products, avoiding the need to explicitly compute $A^{-1}$. These methods are named after Aleksei Nikolaevich Krylov, who first utilized them to study mechanical system oscillations in \cite{krylov1931numerical}, where the subspace takes the following form
\begin{equation}
    \mathcal{K}_{p}(A, \nu):= \text{span}\{\nu , A \nu, A^{2}\nu, \dots, A^{p-1} \nu \},\quad \text{with}\;\; \nu \in \mathbb{C}^{n}.
\end{equation}
Krylov subspace methods for linear systems start from an initial guess $\mathbf{x}_0$, compute the initial residual ${r}_0 = {b} - A \mathbf{x}_0$, and then typically select iterates $\mathbf{x}_1, \mathbf{x}_2, \dots$ such that
\begin{equation}
\mathbf{x}_p - \mathbf{x}_0 \in \mathcal{K}_p(A, \mathbf{r}_0),
\end{equation}
thereby projecting the original problem to one of much smaller dimension. The conjugate gradient (CG) method is the first Krylov subspace method for linear systems, introduced by Hestenes and Stiefel \cite{hestenes1952methods}, with related work by Lanczos \cite{lanczos1950iteration, lanczos1952solution}. Initially viewed as a direct method, CG was later recognized as an iterative method after Reid's observation \cite{reid1971method}. This sparked renewed interest in Krylov subspace methods, leading to the development of widely used approaches such as the Generalized Minimum Residual Method (GMRES) \cite{saad1986gmres} and the Full Orthogonalization Method (FOM) \cite{saad1981krylov}, along with other methods like the Least Squares QR Method (LSQR) \cite{paige1982lsqr} and the Biconjugate Gradient Method (BiCG) \cite{fletcher2006conjugate}. In the present work, we focus exclusively on the GMRES method, which remains one of the most effective tools for solving linear systems. This method calculates a solution $\mathbf{x}_{p} \in \mathbf{x}_{0} + \mathcal{K}_{p}$ by minimizing the residual norm over all vectors in $\mathbf{x}_{0} + \mathcal{K}_{p}$. Specifically, at the $p$-th step, GMRES determines $\mathbf{x}_{p}$ to minimize $\Vert b - A \mathbf{x}_{p} \Vert_{2}$, for all $\mathbf{x}_{p} \in \mathbf{x}_{0} + \mathcal{K}_{p}$. The standard implementation of GMRES relies on the Arnoldi process, as outlined in Algorithm \ref{alg:A15}. This process uses a modified Gram-Schmidt procedure to construct an orthonormal basis within the Krylov subspace.
\begin{algorithm}[H]
\caption{ The Arnoldi process}\label{alg:A15}
\begin{algo}
\INPUT Given a matrix $A$, vector $b$, a positive initial vector $\mathbf{x}_{0}$, and a positive parameter $p$.  
\OUTPUT  An orthogonal basis ${V}_{p}=[v_{1},v_{2},\dots,v_{p}]$ of $\mathcal{K}_{p}$.
\STATE Compute $r_{0}=b-A \mathbf{x}_{0}$, $\beta=\Vert r_{0} \Vert $ and $v_{1}=r_{0}/ \beta$.
\FOR{j=1 \dots p}
 \STATE  $w=A v_{j}$;
 \FOR{i=1 \dots j}
 \STATE $h_{ij}=<w,v_{i}>$; $w=w-v_{i} h_{ij}$;
 \ENDFOR
 \STATE $h_{j+1,j}=\Vert w \Vert$; $v_{j+1}=w/h_{j+1,j}$;
 \ENDFOR 
\end{algo}
\end{algorithm}
Algorithm \ref{alg:A15} generates a set of $p + 1$ orthonormal vectors, denoted by $v_{1}, v_{2}, \dots, v_{p+1}$, and computes the scalars $h_{i,j}$. The first $p$ vectors $(v_1, v_2, \dots, v_p)$ form an orthonormal basis for the Krylov subspace $\mathcal{K}_{p}(A, b)$, which define the matrices $V_{j} = [v_{1}, v_{2}, \dots, v_{j}]$, for $j \in \{p, p+1\}$.  While, from the scalars, we construct the upper Hessenberg matrix $H_{p+1,p} = [h_{i,j}] \in \mathbb{R}^{(p+1)\times p}$.
From these matrices, the recursion formulas for the Arnoldi process can be derived and expressed as a partial Arnoldi decomposition
$$A V_{p}=V_{p+1} H_{p+1,p}.$$
The last relation is applied to compute the GMRES iterate $\mathbf{x}_{p}$ as follows
\begin{equation}\label{eq719}
  \min_{\mathbf{x}_{p}\in \mathcal{K}_{p} } \Vert b- A \mathbf{x}_{p} \Vert=\min_{\mathbf{y} \in \RR^{p}}\Vert b-A V_{p} \mathbf{y}  \Vert=\min_{\mathbf{y} \in \RR^{p}} \Vert  \beta e_{1}- H_{p+1,p} \mathbf{y} \Vert,  
\end{equation}
where  $e_{1}$ denotes the first unit vector of $\RR^{p+1}$. The small minimization problem on the right-hand side of the equation (\ref{eq719}) can be efficiently solved using QR factorization of the upper Hessenberg matrix  $H_{p+1,p}$. This reference \cite{saad2003iterative} provides further details on this technique.

This method has been extensively studied in the literature, with early foundations laid in works such as \cite{brown1990hybrid}. In this paper, we adapt it to efficiently solve the multilinear PageRank problem, offering a novel approach to address its challenges.
\subsection{Nonlinear Krylov algorithms}
We are interested in using a Newton-like iteration scheme to solve the nonlinear system
\begin{equation}
    f(\mathbf{x})=0.
\end{equation}
As discussed previously, at each iteration we must obtain an approximate solution of the following linear system derived from applying the Newton-like method
\begin{equation}\label{eq721}
        J\delta=-f,
    \end{equation}
    where $f$ represents the non-linear function and $J$ its Jacobian matrix, both evaluated at the current iterate. To find this approximate solution, we introduce a correction term $z$.  If $\delta_{0}$ is an initial guess for the true solution of the linear system (\ref{eq721}), then letting $\delta=\delta_{0} + z$, we have the equivalent equation
    \begin{equation}
    Jz=r_{0},
\end{equation}
where $r_{0}=-f-J\delta_{0}$ is the initial residual. In this context, the Krylov subspace is defined as
$$\mathcal{K}_{p} \equiv span\{ r_{0},Jr_{0},\dots, J^{p-1} r_{0} \}.$$ We use GMRES to find an approximate solution, denoted by $\delta_{p}$, and can be expressed as
\begin{equation}
        \delta_{p}=\delta_{0}+z_{p}, \quad \text{with} \; z_{p} \in \mathcal{K}_{p},
    \end{equation}
    such that
    \begin{equation}\label{eq724}
        \min_{\delta_{p}\in \delta_{0}+ \mathcal{K}_{p} } \Vert f+ J \delta_{p}  \Vert_{2}=\min_{\mathbf{z} \in \mathcal{K}_{p} } \Vert r_{0}-J z \Vert_{2}.
    \end{equation}
The following algorithm outlines the main steps of a nonlinear variant of the GMRES algorithm used to compute a multilinear PageRank vector. At each iteration, it generates an orthonormal system of vectors $v_{i}$, with $i=1, 2,\dots, p$  belonging to the subspace $\mathcal{K}_{p}$. Then, the construction of the vector $\delta_{p}$ that satisfies (\ref{eq724}).

\begin{algorithm}[H]
\caption{ Newton-GMRES Method for Multilinear PageRank}\label{alg:A16}
\begin{algo}
\INPUT  Given a transition probability tensor $\PP$, teleportation vector $v$, damping factor \( \alpha \in [0,1) \), maximum number of iterations $K_{max}$, and convergence tolerance $\epsilon$.
\OUTPUT  The multilinear PageRank vector $\mathbf{x}$.
\STATE Initialize $k = 0$, set $\mathbf{x}_0 = v$, and choose an inner GMRES tolerance $\epsilon_{0}$.
\STATE \REPEAT
 \STATE For an initial guess $\delta_{0}$, form $r_{0}=-f-J\delta_{0}$, where $f=f(\mathbf{x}_{k})$ and $J=J_{f}(\mathbf{x}_{k})$. Then, compute $\beta=\|r_{0}\|_{2}$ and $v_{1}=\dfrac{r_{0}}{\beta}$.
 \FOR {$j = 1, \dots, p$}
            \STATE  Form $Jv_j$ and orthogonalize  it against, $v_1, \ldots, v_j$ using the Arnoldi process,
            \STATE compute the residual norm $\rho_j = \|f + J\delta_{j}\|_{2}$ of the solution $\delta_{j}$,
            \STATE  if $\rho_j \leq \epsilon_k$, set $p = j$ and go to (11). 
        \ENDFOR
\STATE Define $\bar{H}_{p} \in \RR^{(p+1)\times p}$, $V_{p} \in \RR^{n \times p}$, find $z_p$ that minimizes $\| \beta e_{1}-\Bar{H}_{p}z \|$ over all vector $z$.
\STATE Compute $\delta_{p} = \delta_{0} + V_{p}z_{p}$.
\STATE  Set $\delta_{k}=\delta_{p}$ and update $\mathbf{x}_{k+1} = \text{proj}(\mathbf{x}_{k} + \delta_k)$.
\STATE $k=k+1$.
\STATE \Until{$\|f(\mathbf{x}_{k})\|_{1} < \epsilon$}.
\end{algo}
\end{algorithm}
One of the main advantages of the Newton-GMRES method is that it does not require the explicit construction of the Jacobian matrix \( J \); instead, it only relies on the action of \( J \) on a vector \( v \). This is particularly useful in large-scale settings, where computing and storing the full Jacobian matrix would be computationally expensive or infeasible. Instead, the matrix-vector product \( J_f(\mathbf{x}) v \) can be efficiently approximated using a finite difference formula:
\begin{equation}
    J_f(\mathbf{x}) v \approx \frac{f(\mathbf{x} + \sigma v) - f(\mathbf{x})}{\sigma},
\end{equation}
where \( \sigma \) is a sufficiently small scalar. This approach, known as the finite difference projection method and first introduced by Brown in~\cite{brown1987local}, significantly reduces memory usage and simplifies implementation. Despite avoiding the explicit Jacobian, it maintains a high-quality approximation of the Newton direction, making it especially suitable for efficiently computing the multilinear PageRank vector. The following algorithm outlines how this technique is integrated into the Newton-GMRES scheme.
\begin{algorithm}[H]
\caption{ The Newton-GMRES with finite difference}\label{alg:A17}
\begin{algo}
\INPUT  Given a transition probability tensor $\PP$, teleportation vector $v$, damping factor \( \alpha \in [0,1) \), maximum number of iterations $K_{max}$, and convergence tolerance $\epsilon$.
\OUTPUT  The multilinear PageRank vector $\mathbf{x}$.
\STATE Initialize $k = 0$, set $\mathbf{x}_0 = v$, and choose an inner GMRES tolerance $\epsilon_{0}$.
\STATE \REPEAT
 \STATE For initial guess $\delta_{0}$, form $q_{0}=\dfrac{(f(\mathbf{x}_{k}+\sigma_{0} \delta_{0})-f(\mathbf{x}_{k}))}{\sigma_{0}}$, and $r_{0}=-f(\mathbf{x}_{k})-q_{0}$. Then, compute $\beta=\|r_{0}\|_{2}$ and $v_{1}=\dfrac{r_{0}}{\beta}$.
 \FOR{$j=1, \dots, p$}
 \STATE Form $\hat{q}_{j}=\dfrac{(f(\mathbf{x}_{k}+\sigma_{j} v_{j})-f(\mathbf{x}_{k}))}{\sigma_{j}}$, and orthogonalize it against, $v_{1},v_{2}, \dots, v_{j}$ using the Arnoldi process,
 \STATE compute $q_{j}=\dfrac{(f(\mathbf{x}_{k}+\sigma \delta_{j})-f(\mathbf{x}_{k}))}{\sigma}$ to estimate $\rho_j = \Vert f(\mathbf{x}_{k}) + q_{j}\Vert_{2}$ for the solution $\delta_{j}$,
            \STATE  if $\rho_j \leq \epsilon_k$, set $p = j$ and go to (9). 
  \ENDFOR
  \STATE Define $\bar{H}_{p} \in \RR^{(p+1)\times p}$, $V_{p} \in \RR^{n \times p}$, find $z_p$ that minimizes $\| \beta e_{1}-\Bar{H}_{p}z \|$ over all vector $z$.
\STATE Compute $\delta_{p} = \delta_{0} + V_{p}z_{p}$.
 \STATE Set $\delta_{k}=\delta_{p}$ and update $\mathbf{x}_{k+1} = \text{proj}(\mathbf{x}_{k} + \delta_p)$.
\STATE $k=k+1$.
\STATE \Until{$\|f(\mathbf{x}_{k})\|_{1} < \epsilon$}.
\end{algo}
\end{algorithm}
\section{ An accelerated Newton-GMRES methods for mutilinear PageRank computations}\label{sec5}
This section discusses how the extrapolation methods can be applied to enhance the convergence of the Newton-GMRES algorithm for computing the multilinear PageRank vector. 
\subsection{ The polynomial extrapolation algorithms}
\noindent Extrapolation methods \cite{sidi2017vector,bentbib2024hosvd,bentbib2024n,bentbib2025einstien} enhance the convergence of slow iterative processes by transforming a given sequence of scalars, vectors, matrices, or, more generally, tensors into a new sequence that converges more rapidly. In the present work, since the solution is vector-valued, we focus on vector-type extrapolation. Among vector methods, two well-known polynomial techniques are the Minimal Polynomial Extrapolation (MPE) and the Reduced Rank Extrapolation (RRE). These methods have attracted considerable attention in various computational settings, as they have proven particularly effective in accelerating PageRank computations \cite{brezinski2005extrapolation}, and this is thanks to the spectral analysis of the Google matrix in \cite{SerraCapizzano2005}.

 \noindent The MPE method, introduced by Cabay and Jackson~\cite{cabay1976polynomial}, approximates the limit of the sequence by using a minimal polynomial. The RRE method, initially proposed by Eddy~\cite{eddy1979extrapolating} and Mesina~\cite{mevsina1977convergence}, improves convergence by reducing the rank of the matrix formed by the sequence. To further improve the practical implementation of these methods, efficient approaches have been developed. Sidi~\cite{sidi1986convergence} proposed optimized versions of the RRE and MPE methods through the use of QR decomposition, offering a more stable and computationally efficient implementation. Jbilou and Sadok~\cite{jbilou1999lu} also introduced an efficient MMPE implementation based on LU decomposition with pivoting.

\noindent The Algorithm \ref{alg:extrapolation} presents the step-by-step procedure for efficiently implementing MPE and RRE methods.
\begin{algorithm}[ht]
        \begin{algo}
        \caption{Algorithms for implementing  MPE and RRE \cite{sidi2012review}}\label{alg:extrapolation}
         \scriptsize
         \INPUT The sequence $\{s^{(0)}, s^{(1)},\dots, s^{(q+1)}\}$
          \OUTPUT $t^{(q)}$
            \STATE Compute $\Delta s^{(i)}=s^{(i+1)}-s^{(i)},\; i=0,1,\dots, q$.
\STATE Set $\Delta S_{j}=[\Delta s^{(0)},\Delta s^{(1)},\dots,\Delta s^{(j)}],\; j=0,1,\dots$ 
                 \STATE Compute the QR-factorization of $\Delta S_{j}$, namely, $\Delta S_{j} = Q_{j} R_{j}$.
                 \begin{itemize}
                     \item \textbf{MPE}: Solve {$R_{q-1}c=-r_{q}$}; where $r_{q}=[r_{0q},r_{1q},\dots,r_{q-1,q}]^{T}$ and $c=[c_{0},c_{1},\dots,c_{q-1}]^{T}$. Then,  calculate $\alpha=\sum\limits_{i=0}^{q} c_{i}$, with $c_{q}=1$, and set {$\gamma_{i}=c_{i}/\alpha,\; i=0,1,\dots,q$}.
                     \item \textbf{RRE}: Solve {$R_{q}^{T} R_{q}d=e$}; where $d=[d_{0},d_{1},\dots,d_{q}]^{T}$ and $e=[1,1,\dots,1]^{T} \in \RR^{q+1}$. Then, set $\lambda=(\sum\limits_{i=0}^{q} d_{i})^{-1}$, and $\gamma=\lambda d$, that is, { $\gamma_{i}=\lambda d_{i},\; i=0,1,\dots,q$}.
                 \end{itemize}
                 \STATE  Compute {$\xi=[\xi_{0},\xi_{1},\dots,\xi_{q-1}]^{T}$} where $\xi_{0}=1-\gamma_{0}$ and $\xi_{j}=\xi_{j-1}-\gamma_{j}$,\; $j=1,\dots,q-1$.
\STATE Compute $t^{(q)}=s^{(0)}+Q_{q-1}(R_{q-1} {\xi}).$
        \end{algo}
    \end{algorithm}
    
 Having discussed these methods, let us now explain how they can be applied to improve the convergence of the Newton-GMRES algorithm. Firstly, we generate $q+2$ iterations to extrapolate the solution, wherein each iteration we utilize either Algorithm \ref{alg:A16} or Algorithm \ref{alg:A17} to generate the sequences. Next, we extrapolate a new approximation using either the MPE or RRE extrapolation techniques. Finally, we normalize to guarantee a stochastic solution, and then we check for convergence, repeating this algorithm until convergence is achieved. We summarize all of these steps in Algorithm \ref{alg:A21}.
\begin{algorithm}[ht] 
\begin{algo}
\caption{ The Newton-GMRES with extrapolation methods}\label{alg:A21}
\INPUT Given a transition probability tensor $\PP$, $\alpha \in [0,1)$, maximum number of iterations $K_{max}$,  number of iterates $q$, and a termination tolerance $\epsilon$.
\OUTPUT  The multilinear PageRank vector $\mathbf{x}$.
\STATE Initialize $k = 0$, choose an inner GMRES tolerance $\epsilon_{0}$, and set $\mathbf{x}_0 = v$ as the initial guess.
\STATE \Do
 \STATE $s_{0}=\mathbf{x}_{k}$
\FOR{$i=0 \dots q+1$}
 \STATE compute $\delta_{i}$ the solution of $J(s_{i})\delta_{i}=-f(s_{i})$ using GMRES
\STATE  $s_{i+1}=s_{i}+\delta_{i}$
\ENDFOR
 \STATE Apply the extrapolation methods MPE or RRE to $s_{0},s_{1},\dots, s_{q+1}$ to compute  the approximation $t^{(q)}$.
\STATE  Set $\mathbf{x}_{k+1}=t^{(q)}$
\STATE  $\mathbf{x}_{k+1} = \text{proj}(\mathbf{x}_{k+1})$.
\STATE $k=k+1$.
 \STATE \Until{$\|f(\mathbf{x}_{k})\|_{1} < \epsilon$}.
\end{algo}
\end{algorithm}
\subsection{The Anderson acceleration method}

The Anderson acceleration (AA) method~\cite{anderson1965iterative} is a technique designed to speed up the convergence of fixed-point iterations.  
Instead of relying solely on the most recent iterate to compute the next approximation, AA combines information from several previous iterates to produce an improved estimate.  
The combination coefficients are determined by minimizing the residual norm, which often leads to a substantial reduction in the number of iterations.

\noindent In the context of Newton-GMRES, each update can be expressed as
\[
\mathbf{x}_{k+1} = \mathbf{x}_k + \delta_k,
\]
where $\delta_k$ solves the linear system \( J(\mathbf{x}_k) \delta_k = -f(\mathbf{x}_k) \).  
This formulation can be interpreted as a fixed-point mapping, and applying AA to this sequence can accelerate its convergence. 
\noindent In its general form, AA constructs the next iterate by taking a linear combination of the last \(t\) updates
\[
\delta_{k-j} = \mathbf{x}_{k-j+1} - \mathbf{x}_{k-j}, \quad j = 0, 1, \dots, t-1,
\]
and choosing the combination coefficients to minimize the residual norm of the new iterate.  The parameter \(t \geq 1\) determines how many past iterates are used, with larger values incorporating more history and potentially improving convergence at the cost of increased computation.

In this work, we restrict ourselves to the simplest case \(t = 1\), which already provides a noticeable performance improvement in practice.  
The corresponding procedure is presented below.
\begin{algorithm}[H]
\caption{Newton-GMRES with Anderson Acceleration}\label{alg:A19}
\begin{algo}
\INPUT A transition probability tensor $\mathcal{P}$, damping factor $\alpha \in [0,1)$, maximum number of iterations $K_{\max}$, and a convergence tolerance $\epsilon$.
\OUTPUT The multilinear PageRank vector $\mathbf{x}$.
\STATE Initialize $k = 0$, set $\mathbf{x}_0 = v$, and an inner GMRES tolerance $\epsilon_0$.
\STATE  \REPEAT
            \STATE Solve $J_{f}\left(\mathbf{x}_{k}\right) \delta_k = -f\left(\mathbf{x}_{k}\right)$, using GMRES
            
    \STATE Compute $\gamma^k = \frac{(\delta_k, \delta_k - \delta_{k-1})}{\| \delta_k - \delta_{k-1} \|^2}$
    \STATE Update:
    \[
    \mathbf{x}_{k+1} = \mathbf{x}_k + \delta_k - \gamma^k \left[(\mathbf{x}_k - \mathbf{x}_{k-1}) + (\delta_k - \delta_{k-1})\right]
    \]
    \STATE  $\mathbf{x}_{k+1} = \text{proj}(\mathbf{x}_{k+1})$
    \STATE $k = k + 1$
\STATE \Until{$\|f(\mathbf{x}_k)\|_1 < \epsilon$}
\end{algo}
\end{algorithm}

\section{Numerical Experiments}\label{sec5k}
In this section, we present a comprehensive set of numerical experiments to assess the performance of our Newton-GMRES-based methods for solving the multilinear PageRank problem. We evaluate the convergence behavior, computational efficiency, and robustness of the baseline Newton-GMRES (NG) algorithm, as well as its extrapolated variants incorporating Minimal Polynomial Extrapolation (NG-MPE) and Reduced Rank Extrapolation (NG-RRE). For comparison, we also include several established Newton-type solvers: the classical Newton method (N), the Modified Newton method (NM), and Newton-GMRES with Anderson acceleration (NA).

Our experiments are organized into three categories. First, we use a set of benchmark tensors proposed by Gleich et al.~\cite{gleich2015multilinear}, which are widely employed in the literature. Second, we generate synthetic tensors using the stochastic construction of~\cite{guo2018modified}, varying the dimension $n$ to study scalability with respect to problem size. Third, we evaluate the methods on real-world networks from~\cite{grindrod2016comparison,california2006}, covering a broad spectrum of graph structures and sparsity patterns.

\noindent  
For a fair comparison, all algorithms are initialized with the same starting vector $\mathbf{x}_{0} = \mathbf{v} = \frac{e}{n}$, where $e = [1, \dots, 1]^T$. The maximum number of iterations is set to $k_{\text{max}} = 1000$. For the extrapolation methods, we consider $q \in \{3,4,5\}$, while for the Modified Newton method we fix the parameter $n_i = 4$. The stopping tolerance is set to $\epsilon = 10^{-15}$, and the inner GMRES tolerance is $\epsilon_{0} = 10^{-14}$ with $p = 40$.  

\noindent All computations were performed in MATLAB R2020a on an Intel\textsuperscript{\textregistered} Core\textsuperscript{\texttrademark} i7-1065G7 CPU @ 1.30~GHz (up to 1.50~GHz), with 16~GB of RAM.

\subsection*{Example 1: Benchmark problems}
We conducted experiments across the entire benchmark suite comprising 29 third-order stochastic tensors with dimensions $n = 3, 4, 6$, as originally introduced in~\cite{gleich2015multilinear}. Each tensor $\mathcal{P}_{i,j}$ is reshaped into a column-stochastic matrix $R_{i,j} \in \mathbb{R}^{n \times n^2}$, where $i$ denotes the tensor size and $j$ the problem index. The multilinear PageRank equation under consideration is given by
\[
f(\mathbf{x}) = \alpha R_{i,j} (\mathbf{x} \otimes \mathbf{x}) + (1 - \alpha) v-\mathbf{x}=0
\]

To illustrate the overall performance trends, we report results for two representative instances: $R_{3,5}$ and $R_{4,8}$. These examples capture the general behavior observed across the benchmark set when varying the damping factor $\alpha$. In particular, we evaluate both the theoretically favorable regime $\alpha < \frac{1}{m-1}$, where convergence of Newton-type methods is expected, and more challenging settings with larger values of $\alpha$. In all cases, the extrapolated Newton-GMRES variants (NGMRES-MPE and NGMRES-RRE) demonstrate superior convergence speed and computational efficiency, confirming their robustness across a wide spectrum of problem instances. The corresponding results are summarized in Tables~\ref{tab1:iterations},~\ref{tab1:time},~\ref{tab2:iterations},  and~\ref{tab2:time}, and to clearly visualize the performance, we plot some of these results in the following Figures~\ref{fig:R35_results} and~\ref{fig:R48_results}.

\begin{table}[H]
\centering
\caption{Number of iterations for Problem $R_{3,5}$.}
\label{tab1:iterations}
\begin{tabular}{lccccc}
\toprule
$\alpha$ & Newton & Newton-GMRES & NGMRES-MPE & NGMRES-RRE & Newton-Anderson \\
\midrule
0.490 & 4   & 5   & \textbf{2}   & \textbf{2}   & 9   \\
0.600 & 4   & 5   & \textbf{3}   & \textbf{3}   & 10  \\
0.700 & 5   & 6   & \textbf{4}   & \textbf{4}           & 11  \\
0.800 & 6   & 7   & \textbf{3}   & \textbf{3}   & 12  \\
0.850 & 6   & 7   & \textbf{3}   & \textbf{4}          & 14  \\
0.900 & 5   & 6   & \textbf{2}   & \textbf{2}   & 14  \\
0.950 & 67  & 68  & \textbf{2}   & \textbf{2}            & 36  \\
0.990 & 27  & 28  & \textbf{2}  & \textbf{2}           & 282 \\
0.999 & 223 & 144 &\textbf{9}           & \textbf{8}  & 89  \\
\bottomrule
\end{tabular}
\end{table}

\begin{table}[H]
\centering
\caption{CPU time (in seconds) for  Problem $R_{3,5}$.}
\label{tab1:time}
\begin{tabular}{lccccc}
\toprule
$\alpha$ & Newton & Newton-GMRES & NGMRES-MPE & NGMRES-RRE & Newton-Anderson \\
\midrule
0.490 & 0.0447  & 0.0267  & 0.0130  & \textbf{0.0095}  & 0.0745  \\
0.600 & 0.0037  & 0.0203  & \textbf{0.0030}  & 0.0099  & 0.0104  \\
0.700 & 0.0026  & \textbf{0.0017}  & 0.0077  & 0.0038  & 0.0024  \\
0.800 & 0.0011  & 0.0025  & 0.0013  & \textbf{0.0008}  & 0.0036  \\
0.850 & 0.0098  & 0.0038  & \textbf{0.0009}  & 0.0013  & 0.0105  \\
0.900 & 0.0005  & 0.0004  & \textbf{0.0001}  & 0.0002  & 0.0029  \\
0.950 & 0.0854  & 0.0801  & 0.0005  & \textbf{0.0004}  & 0.0112  \\
0.990 & 0.0126  & 0.0169  & \textbf{0.0005}  & \textbf{0.0005}  & 0.7138  \\
0.999 & 1.7639  & 0.2566  & 0.0097 & \textbf{0.0066}  & 0.0739  \\
\bottomrule
\end{tabular}
\end{table}
\begin{table}[H]
\centering
\caption{Number of iterations for Problem $R_{4,8}$.}
\label{tab2:iterations}
\begin{tabular}{lccccc}
\toprule
$\alpha$ & Newton & Newton-GMRES & NGMRES-MPE & NGMRES-RRE & Newton-Anderson \\
\midrule
0.490 & 4   & 5   & \textbf{2}   & \textbf{2}   & 8   \\
0.600 & 4   & 5   & \textbf{3}   & \textbf{2}   & 9  \\
0.700 & 5   & 6   & \textbf{3}   & \textbf{3}           & 10  \\
0.800 & 5   & 7   & \textbf{3}   & \textbf{3}   & 12  \\
0.850 & 6   & 7   & \textbf{3}   & \textbf{3}          & 11  \\
0.900 & 6   & 7   & \textbf{3}   & \textbf{3}   & 11  \\
0.950 & 6  & 7  & \textbf{3}   & \textbf{3}            & 11  \\
0.990 & 6  & 7  & \textbf{3}  & \textbf{3}           & 11 \\
0.999 & 6 & 7 &\textbf{3}           & \textbf{3}  & 11  \\
\bottomrule
\end{tabular}
\end{table}

\begin{table}[H]
\centering
\caption{CPU time (in seconds) for Problem $R_{4,8}$.}
\label{tab2:time}
\begin{tabular}{lccccc}
\toprule
$\alpha$ & Newton & Newton-GMRES & NGMRES-MPE & NGMRES-RRE & Newton-Anderson \\
\midrule
0.490 & 0.046952 & 0.026037 & 0.014426 & 0.006120 & 0.038233 \\
0.600 & 0.038113 & 0.014407 & 0.001009 & 0.000803 & 0.046596 \\
0.700 & 0.048978 & 0.030498 & 0.012748 & 0.001470 & 0.046808 \\
0.800 & 0.078088 & 0.057764 & 0.003175 & 0.003762 & 0.061970 \\
0.850 & 0.057822 & 0.033389 & 0.001201 & 0.000936 & 0.056272 \\
0.900 & 0.060610 & 0.036466 & 0.001896 & 0.001212 & 0.055810 \\
0.950 & 0.119066 & 0.035245 & 0.001993 & 0.002064 & 0.058882 \\
0.990 & 0.058934 & 0.035319 & 0.002068 & 0.001194 & 0.055587 \\
0.999 & 0.068145 & 0.035257 & 0.001849 & 0.001368 & 0.053926 \\
\hline
\bottomrule
\end{tabular}
\end{table}
\begin{figure}[H]
    \centering
    \begin{subfigure}[b]{0.48\textwidth}
        \includegraphics[width=\textwidth]{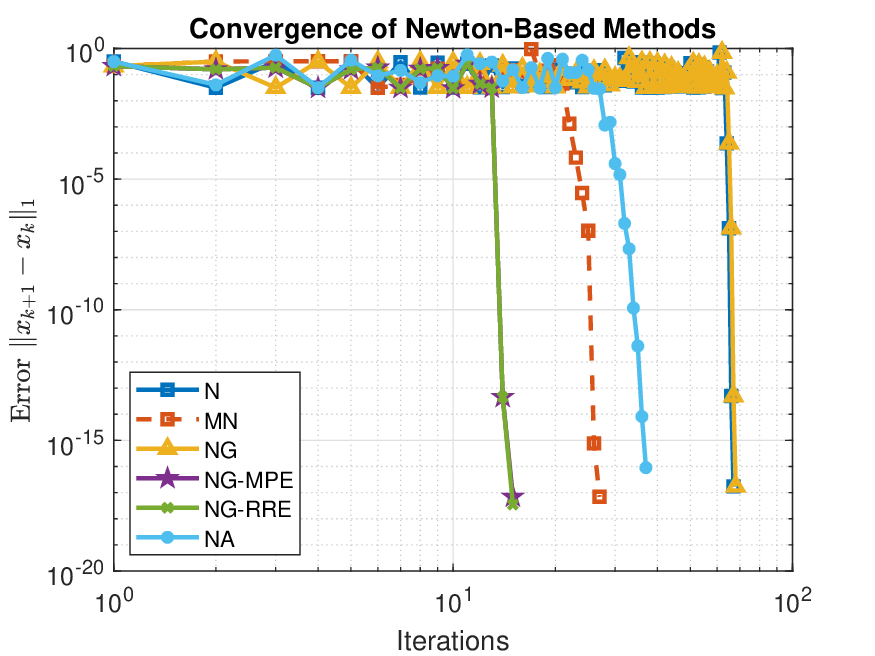}
        \caption{The problem $R_{3,5}$}
        \label{fig:R35_results}
    \end{subfigure}
    \hfill
    \begin{subfigure}[b]{0.48\textwidth}
        \includegraphics[width=\textwidth]{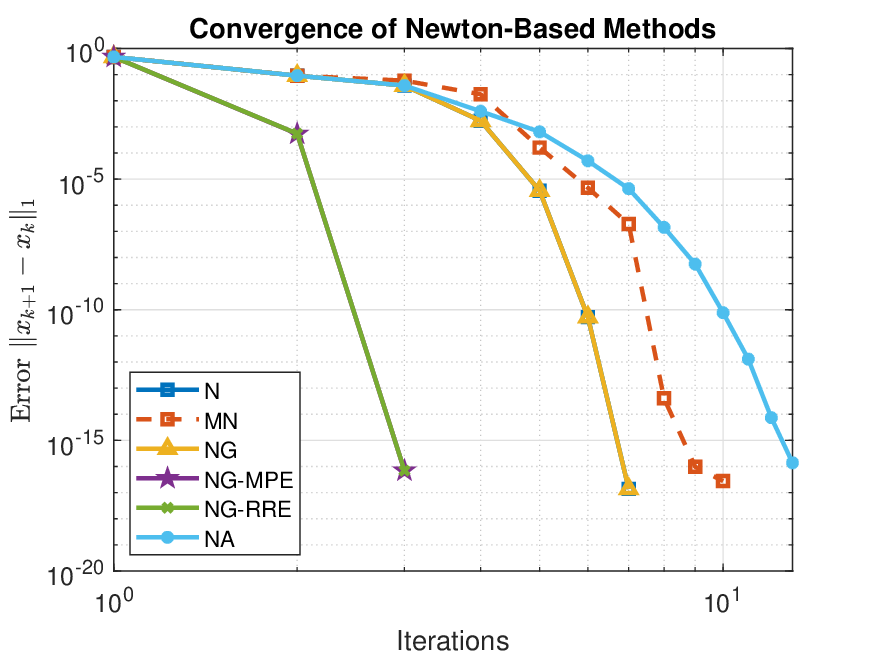}
        \caption{The problem $R_{4,8}$}
        \label{fig:R48_results}
    \end{subfigure}
    \caption{ The performance of our proposed methods for the problems $R_{3,5}$ and $R_{4,8}$ with $\alpha = 0.95$.}
    \label{fig1}
\end{figure}

To further assess the accuracy and efficiency of the proposed methods, we report performance profiles in Figures~\ref{fig1},~\ref{fig2},~\ref{fig3}, and ~\ref{fig4}, based on iteration count and CPU time. These profiles allow for a comprehensive comparison across all tested methods. For each test problem, we compute the ratio between the performance of a given method and the best performance obtained on that problem, whether in terms of the number of iterations or total runtime. Then, for each value of $\tau \geq 1$, the profile indicates the fraction of problems for which the method performs within a factor $\tau$ of the best. This approach provides a standardized and quantitative view of both the efficiency and robustness of the solvers. The methodology follows the framework introduced by Dolan and Moré~\cite{dolan2002benchmarking}.

\begin{figure}[H]
    \centering
    \begin{subfigure}[b]{0.48\textwidth}
        \includegraphics[width=\textwidth]{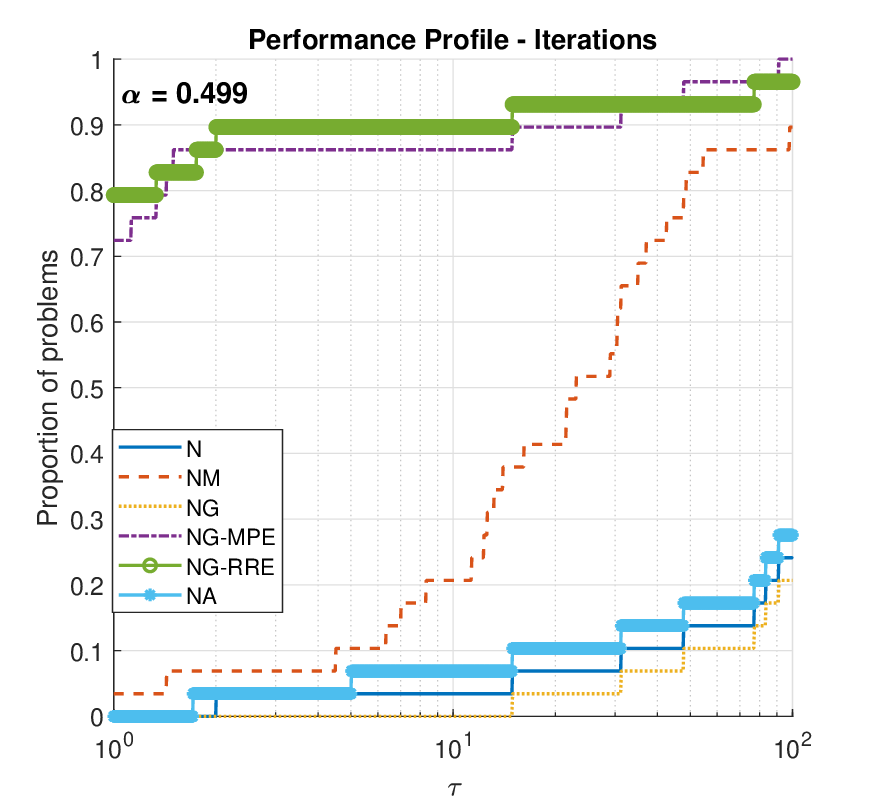}
        \caption{Number of iterations}
        \label{fig1:iter}
    \end{subfigure}
    \hfill
    \begin{subfigure}[b]{0.48\textwidth}
        \includegraphics[width=\textwidth]{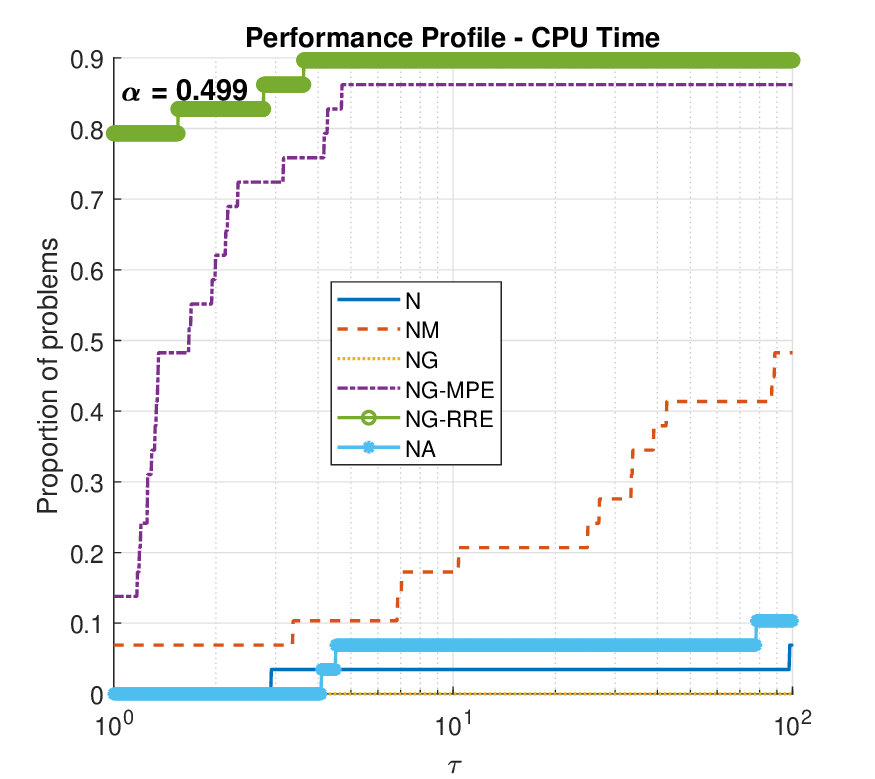}
        \caption{CPU time (in seconds)}
        \label{fig1:time}
    \end{subfigure}
    \caption{ Performance profiles for the 29 benchmark tensors with $\alpha = 0.499$.}
    \label{fig1k}
\end{figure}
\begin{figure}[H]
    \centering
    \begin{subfigure}[b]{0.48\textwidth}
        \includegraphics[width=\textwidth]{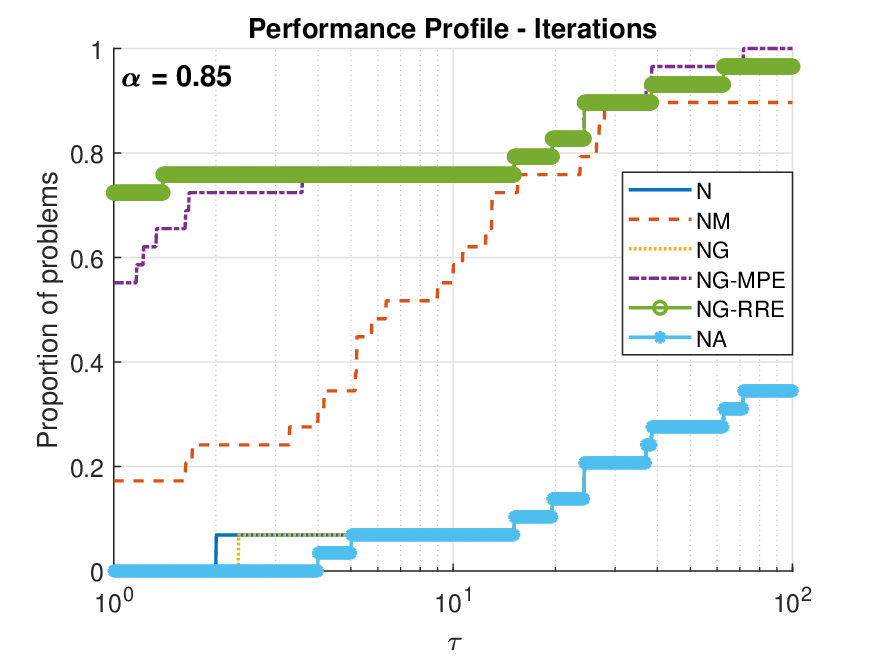}
        \caption{Number of iterations}
        \label{fig2:iter}
    \end{subfigure}
    \hfill
    \begin{subfigure}[b]{0.48\textwidth}
        \includegraphics[width=\textwidth]{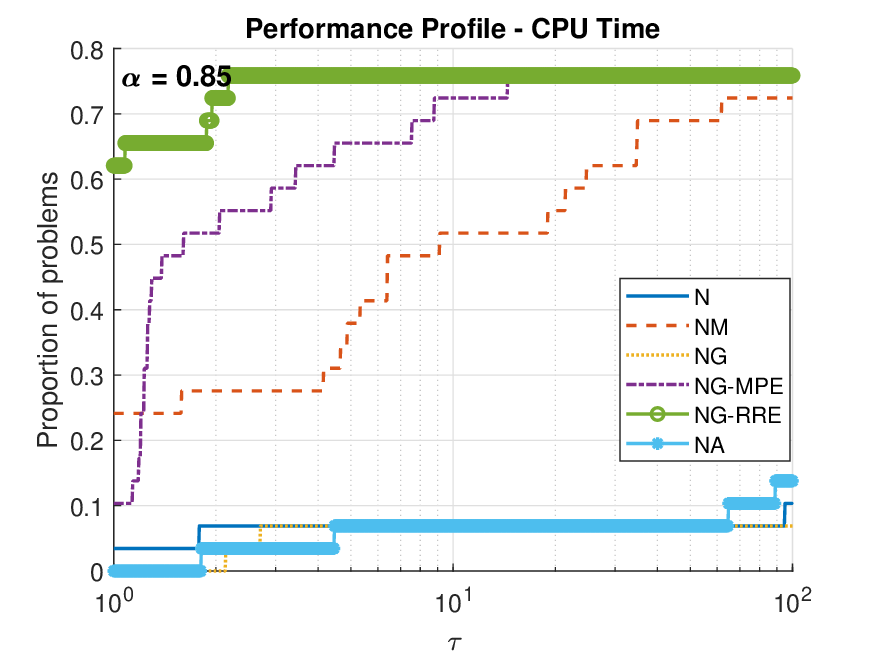}
        \caption{CPU time (in seconds)}
        \label{fig2:time}
    \end{subfigure}
    \caption{ Performance profiles for the 29 benchmark tensors with $\alpha = 0.85$.}
    \label{fig2}
\end{figure}
\begin{figure}[H]
    \centering
    \begin{subfigure}[b]{0.48\textwidth}
        \includegraphics[width=\textwidth]{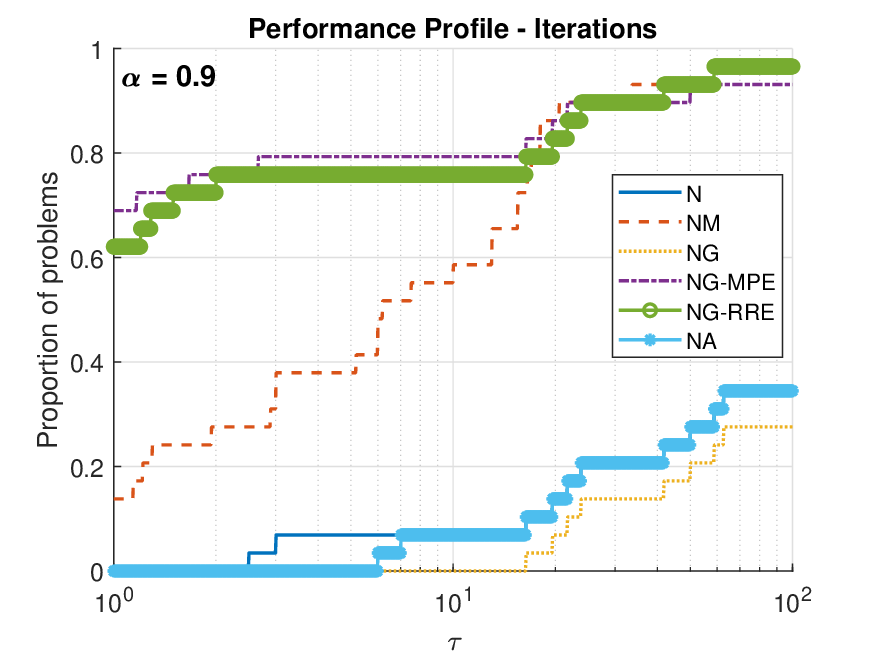}
        \caption{Number of iterations}
        \label{fig3:iter}
    \end{subfigure}
    \hfill
    \begin{subfigure}[b]{0.48\textwidth}
        \includegraphics[width=\textwidth]{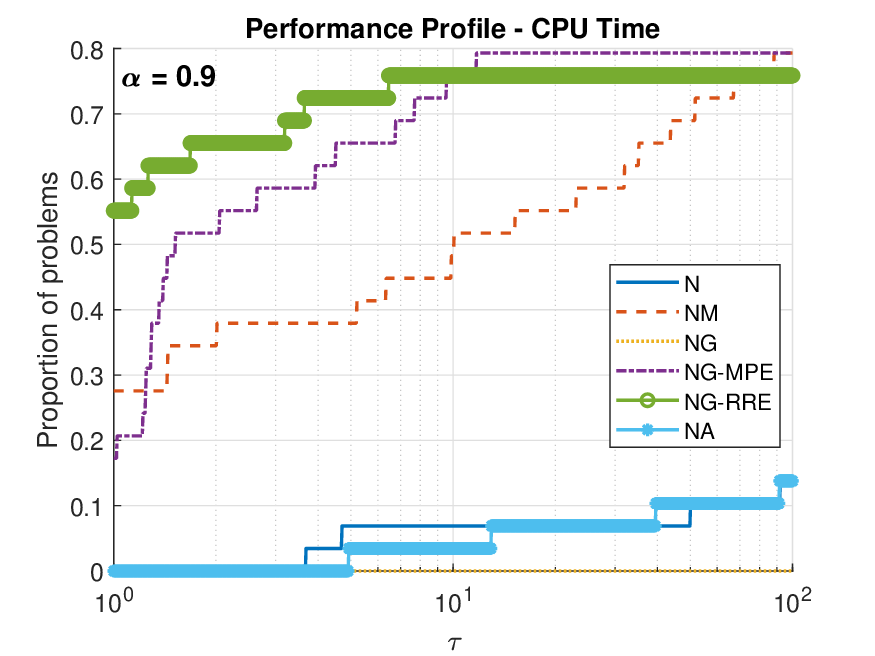}
        \caption{CPU time (in seconds)}
        \label{fig3:time}
    \end{subfigure}
    \caption{ Performance profiles for the 29 benchmark tensors with $\alpha = 0.90$.}
    \label{fig3}
\end{figure}
\begin{figure}[H]
    \centering
    \begin{subfigure}[b]{0.48\textwidth}
        \includegraphics[width=\textwidth]{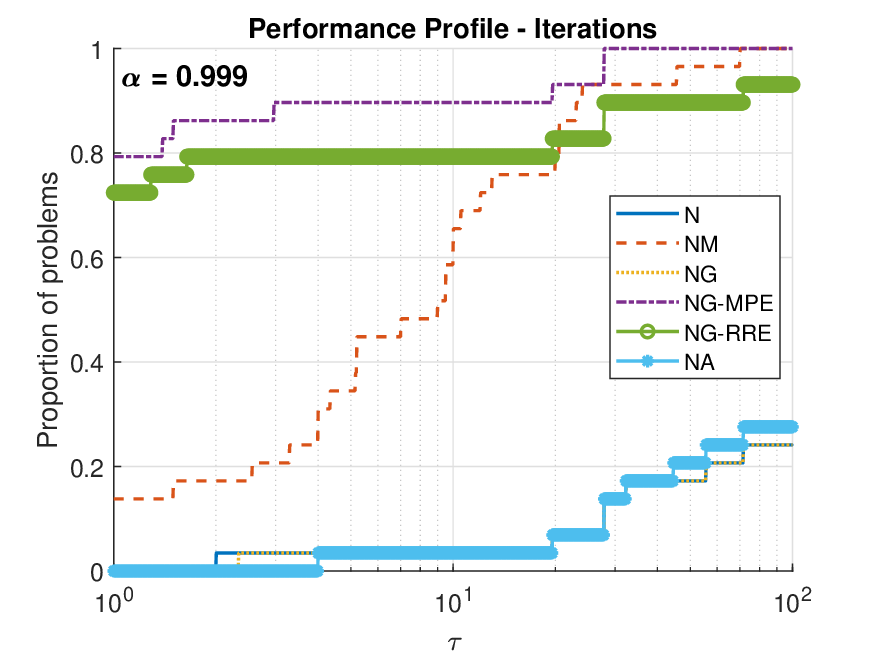}
        \caption{Number of iterations}
        \label{fig4:iter}
    \end{subfigure}
    \hfill
    \begin{subfigure}[b]{0.48\textwidth}
        \includegraphics[width=\textwidth]{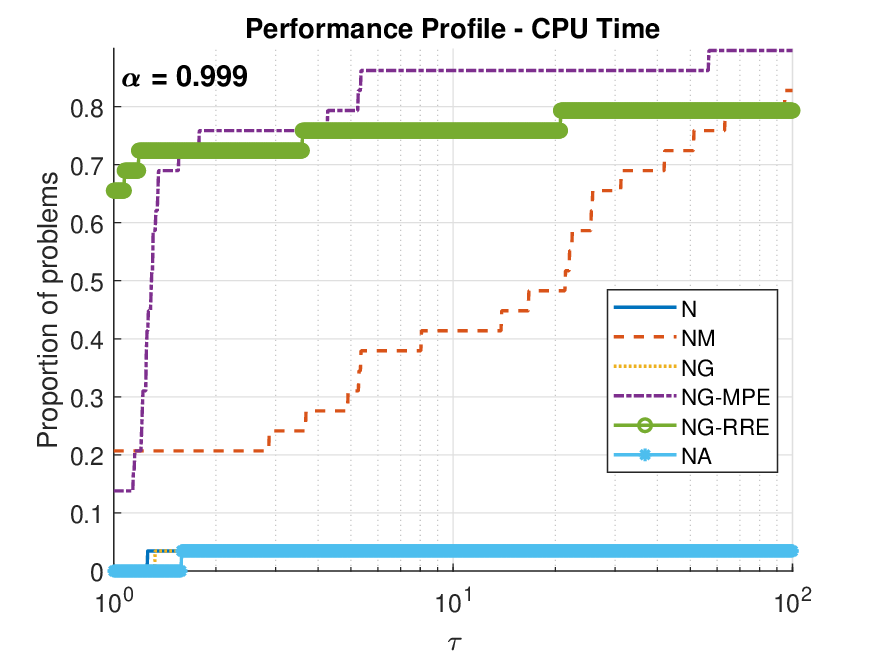}
        \caption{CPU time (in seconds)}
        \label{fig4:time}
    \end{subfigure}
    \caption{ Performance profiles for the 29 benchmark tensors with $\alpha = 0.999$.}
    \label{fig4}
\end{figure}

Tables~\ref{tab1:time}, \ref{tab1:iterations}, \ref{tab2:time}, and \ref{tab2:iterations} reveal that as $\alpha$ approaches~1, both the standard Newton method and Newton-GMRES exhibit a sharp increase in iteration count and runtime. In contrast, the extrapolated variants of Newton-GMRES, namely NG-MPE and NG-RRE, maintain consistently low iteration counts and CPU times even for challenging cases such as $\alpha = 0.999$ that demonstrate improved stability and efficiency across all tested values of~$\alpha$.

Figures~\ref{fig:R35_results} and~\ref{fig:R48_results} further confirm these findings. The residual norm plots clearly show that NG-MPE and NG-RRE converge much faster than the classical approaches, particularly in stiff regimes. This highlights the advantage of incorporating extrapolation into the Newton-GMRES framework to accelerate convergence without sacrificing robustness.

In addition, the performance profiles in Figures~\ref{fig1k}--\ref{fig4} provide a comprehensive comparison of the tested methods over the full benchmark suite. These profiles illustrate the proportion of problems solved within a given factor of the best method, in terms of both iteration count and CPU time. NG-MPE and NG-RRE dominate consistently, with their performance gap increasing as $\alpha$ grows. For $\alpha = 0.90$ and $\alpha = 0.999$, the superiority of these methods becomes particularly evident, confirming their scalability and robustness in stiff scenarios. These results underscore the practical benefit of coupling Krylov-based solvers with polynomial extrapolation for efficient large-scale multilinear PageRank computation.

\subsection*{Example 2: Synthetic Problems}

To further evaluate the scalability of the proposed solvers, we consider a series of synthetic problems constructed from randomly generated third-order stochastic tensors. These artificial datasets simulate networks with controlled structure and allow us to examine how the performance of the algorithms evolves as the problem size increases.

Each tensor is generated by constructing a random matrix $R \in \mathbb{R}^{n \times n^2}$ with entries uniformly drawn from the interval $[0,1]$, followed by a column-wise normalization to ensure the stochasticity condition required for the multilinear PageRank model. The teleportation vector $v \in \mathbb{R}^n$ is initialized as the uniform vector and normalized to form a probability distribution. The resulting tensor is then reshaped into its mode-1 matricization for use in computations. The detailed generation procedure is given in Algorithm~\ref{alg:page}.

\begin{algorithm}[H]
\begin{algo}
\caption{Generation of a PageRank tensor $R$}
\label{alg:page}
\scriptsize
\INPUT Integer $n$
\OUTPUT Teleportation vector $v \in \RR^n$ and normalized matrix $R \in \RR^{n \times n^2}$
\STATE Initialize $v \leftarrow e$ \COMMENT{$e$ is the vector of all ones}
\STATE Initialize a random matrix $R \in \RR^{n \times n^2}$ with entries uniformly drawn from $[0,1]$
\STATE Compute $s \leftarrow v^\top R$
\FOR{$i = 1$ \To $n^2$}
    \IF{$s(i) > 0$}
        \STATE $R(:,i) \leftarrow R(:,i) / s(i)$
    \ELSE
        \STATE $R(:,i) \leftarrow v$
    \ENDIF
\ENDFOR
\STATE Normalize $v \leftarrow v / \|v\|_1$
\end{algo}
\end{algorithm}

We conduct experiments for six different tensor sizes: $n = 100$, $300$, $500$, and $800$. For each case, the algorithms are tested across a wide range of damping factors, with $\alpha$ varying from $0.70$ to $0.999$.

In Figures~\ref{fig:100}--\ref{fig:500}, we plot the evolution of the relative residual norm $\|\mathbf{x}_{k+1}-\mathbf{x}_k\|_1$ as a function of the iteration number for several values of $\alpha$. This representation allows us to observe how the convergence behavior of the methods changes as $\alpha$ increases. In particular, we aim to assess the impact of large damping factors on both convergence rate and numerical stability. The results show that, while the convergence tends to slow down when $\alpha$ approaches 1, the accelerated variants still exhibit stable and efficient performance even in stiff regimes.

\begin{figure}[H]
    \centering

    \begin{subfigure}[b]{0.45\textwidth}
        \centering
        \includegraphics[width=\textwidth]{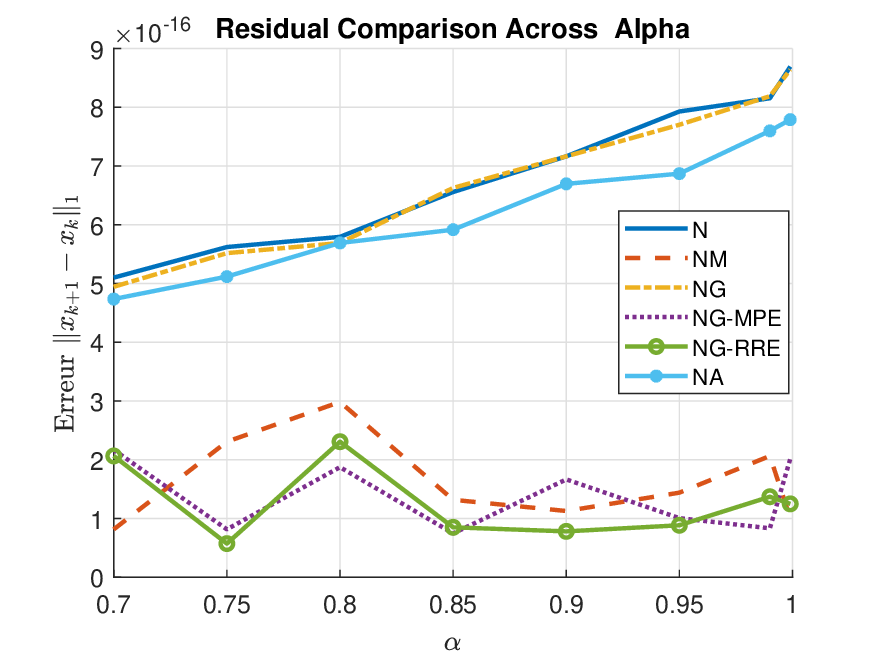}
        \caption{$n = 100$}
        \label{fig:100}
    \end{subfigure}
    \hfill
    \begin{subfigure}[b]{0.45\textwidth}
        \centering
        \includegraphics[width=\textwidth]{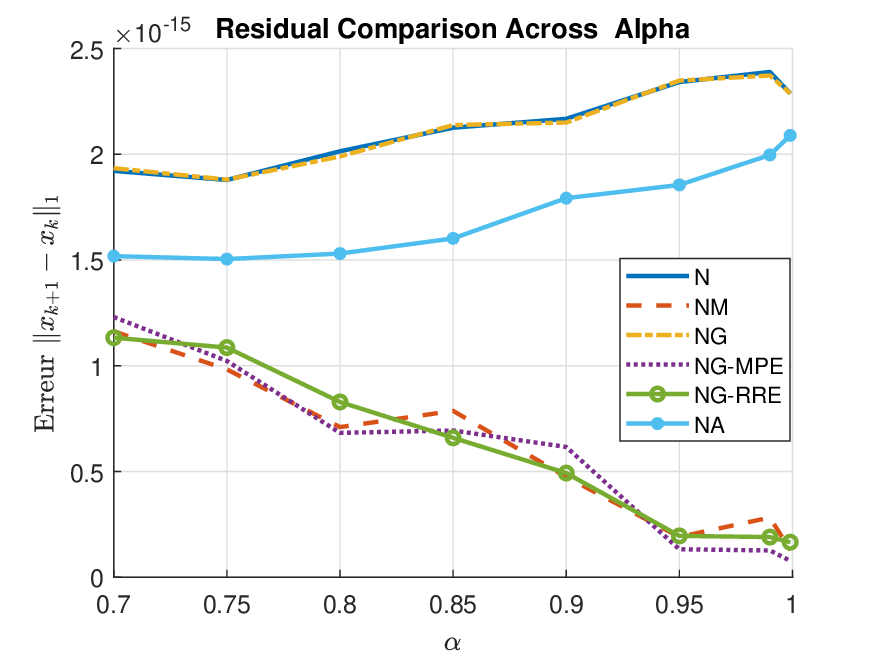}
        \caption{$n = 300$}
        \label{fig:300}
    \end{subfigure}

    \vskip\baselineskip

    \begin{subfigure}[b]{0.45\textwidth}
        \centering
        \includegraphics[width=\textwidth]{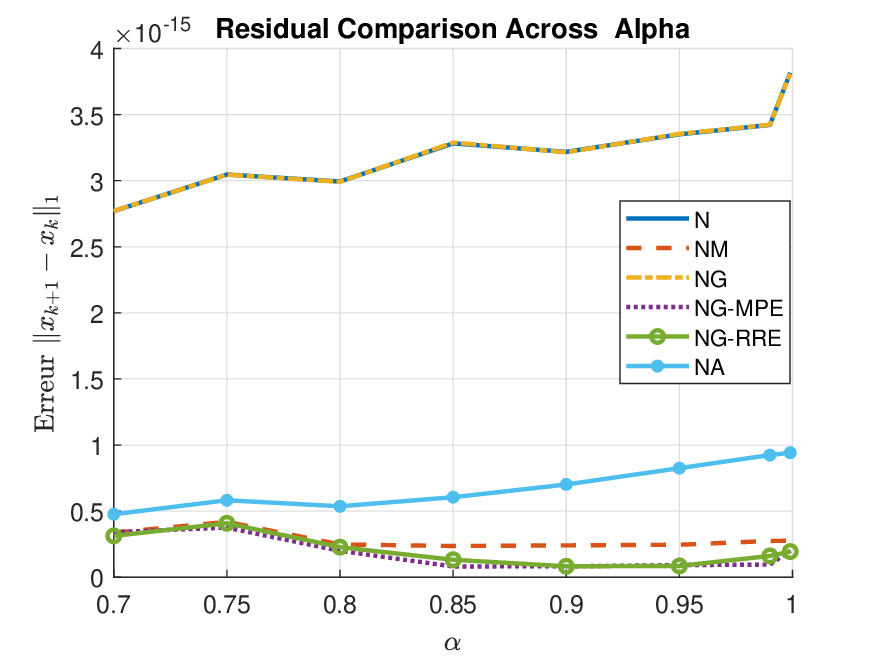}
        \caption{$n = 500$}
        \label{fig:500}
    \end{subfigure}
    \hfill
    \begin{subfigure}[b]{0.45\textwidth}
        \centering
        \includegraphics[width=\textwidth]{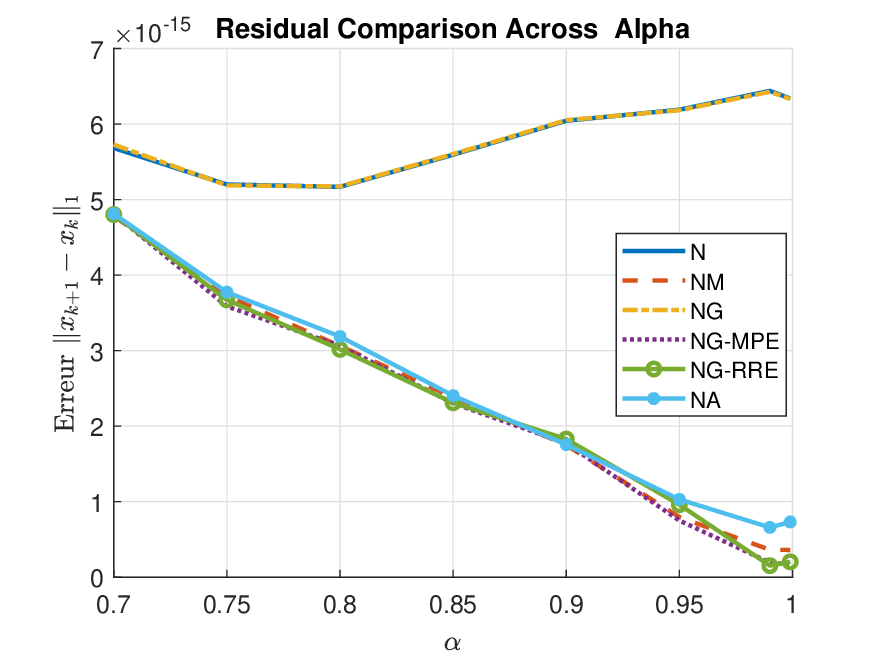}
        \caption{$n = 800$}
        \label{fig:800}
    \end{subfigure}

    \caption{Convergence behavior for different values of $n$ and multiple values of $\alpha$.}
    \label{fig:convergence_all}
\end{figure}

\subsection*{Example 3: Real-World Problems}

To evaluate the practical applicability of the proposed algorithms, we consider a set of real-world directed networks, reformulated as third-order PageRank problems. This reformulation follows a tensor-based modeling approach inspired by~\cite{gleich2015multilinear,benson2015tensor}, which captures higher-order dynamics through local three-node cycles.

Given a directed graph $\mathcal{G} = (\mathcal{V}, \mathcal{E})$ with $|\mathcal{V}| = n$ nodes and $|\mathcal{E}| = l$ edges, we construct the third-order adjacency tensor $\mathcal{T} = [t_{ijk}] \in \mathbb{R}^{n \times n \times n}$ such that

\[
t_{ijk} = 
\begin{cases}
1 & \text{if } i,j,k \text{ forms a directed 3-cycle} \begin{tikzpicture}[baseline={([yshift=-.6ex]current bounding box.center)}, scale=0.4, every node/.style={circle, draw, inner sep=1pt, font=\scriptsize}]
  \node (i) at (90:1) {$i$};
  \node (j) at (210:1) {$j$};
  \node (k) at (330:1) {$k$};
  \draw[->, >=stealth] (i) -- (j);
  \draw[->, >=stealth] (j) -- (k);
  \draw[->, >=stealth] (k) -- (i);
\end{tikzpicture}, \\
0 & \text{otherwise}.
\end{cases}
\]
This tensor encodes triadic cyclic interactions that frequently occur in many real-world networks.

The tensor $\mathcal{T}$ is then unfolded along mode one to form a matrix $Q \in \mathbb{R}^{n \times n^2}$. Each column of $Q$ is normalized so that its entries sum to one whenever the column contains at least one nonzero value; otherwise, the column is replaced by a fixed stochastic teleportation vector $\mathbf{v} \in \mathbb{R}^n$.  

In parallel, we build the standard first-order random-walk component from the adjacency matrix $A$ of $\mathcal{G}$. Let $D$ be the diagonal degree matrix with $D_{ii} = \sum_j A_{ij}$, and define
\[
P = D^{\dagger} A,
\]
where $D^{\dagger}$ denotes the Moore–Penrose pseudoinverse of $D$.

To account for dangling nodes and preserve stochasticity, we use the operator
\[
\mathrm{dang}(B) = e_r^\top - e_n^\top B,
\]
for any $B \in \mathbb{R}^{n \times r}$, where $e_k$ is the all-ones vector of dimension $k$.  

Finally, we construct the mode-1 unfolding of the stochastic tensor $R$ as a convex combination of higher-order and first-order components:
\[
R = \gamma \left[ Q + \mathbf{v} \cdot \mathrm{dang}(Q) \right] 
  + (1 - \gamma) \left[ P + \mathbf{v} \cdot \mathrm{dang}(P) \right] e_n^\top,
\]
where $\gamma \in [0,1]$ controls the relative contribution of each component.  

The multilinear PageRank vector $\mathbf{x}$ is then defined as the unique solution of the fixed-point equation
\[
f(\mathbf{x}) = \alpha R(\mathbf{x} \otimes \mathbf{x}) + (1 - \alpha)\mathbf{v}-\mathbf{x}=0,
\]
with damping factor $\alpha \in [0,1)$.

Table~\ref{tab:networks} summarizes the key characteristics of the real-world networks used in our experiments. Each dataset is described by its number of nodes and edges, as well as the total number of directed three-node cycles (non-zero entries in the tensor $\mathcal{T}$), which reflect the richness of higher-order interactions captured by the model.

\begin{table}[H]
\centering
\resizebox{0.85\textwidth}{!}{
\begin{tabular}{@{}lllll@{}}
\toprule
Network & Reference & Nodes ($n$) & Edges ($l$) & 3-Cycles \\
\midrule
Edinburgh & \cite{grindrod2016comparison} & 1645 & 4292 & 1404 \\
Nottingham & \cite{grindrod2016comparison} & 2066 & 6310 & 6162 \\
Cardiff & \cite{grindrod2016comparison} & 2685 & 8888 & 15186 \\
Bristol & \cite{grindrod2016comparison} & 2892 & 9076 & 7722 \\
California & \cite{california2006} & 9664 & 16150 & 537 \\
\bottomrule
\end{tabular}
}
\caption{Summary of real-world networks used to construct PageRank tensors.}
\label{tab:networks}
\end{table}
In the following figures, we show the convergence behavior (left) and the total execution time (right) of the tested methods on the real-world networks listed in Table~\ref{tab:networks}, using a damping factor of $\alpha = 0.99$.

The results show similar patterns across all datasets. The extrapolated Newton-GMRES methods (NG-MPE and NG-RRE) reduce the residual very well, but the standard Newton-GMRES (NG) is usually the fastest in terms of execution time, especially on large networks. This shows that solving the Newton system approximately using Krylov subspace methods and without extrapolation can be both fast and accurate.

Newton-Anderson (NA) is also an effective method for large problems. It converges quickly, with fewer iterations and acceptable running time. Its performance improves as the network size increases.

As shown in Figures~\ref{fig:edinburgh_results} to~\ref{fig:bristole_results}, NG offers a good balance between speed and precision, while NG-MPE and NG-RRE reach more accurate solutions with slightly more cost. NA also performs well on large networks.

On the other hand, the Modified Newton (NM) method is the slowest and requires the most iterations in all cases. These results confirm that classical Newton methods are less suitable for large problems, while accelerated Newton-Krylov approaches especially NG, NG-MPE, NG-RRE and NA are more efficient and scalable for computing the multilinear PageRank vector on real data.

\begin{figure}[H]
    \centering

    \begin{subfigure}[b]{0.48\textwidth}
        \centering
        \includegraphics[width=\textwidth]{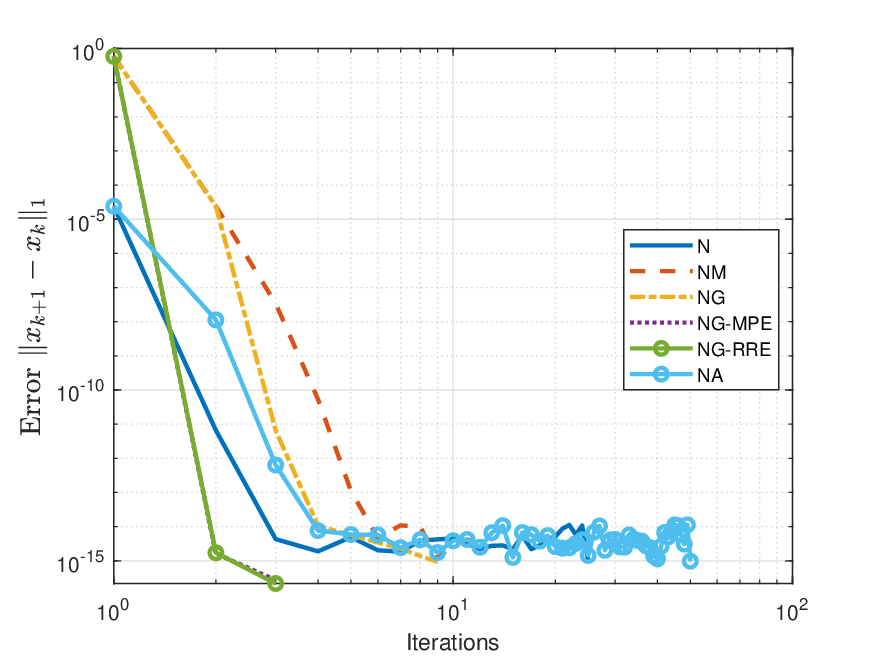}
        \caption{Convergence comparison}
        \label{fig:edinburgh_residual}
    \end{subfigure}
    \hfill
    \begin{subfigure}[b]{0.48\textwidth}
        \centering
        \includegraphics[width=\textwidth]{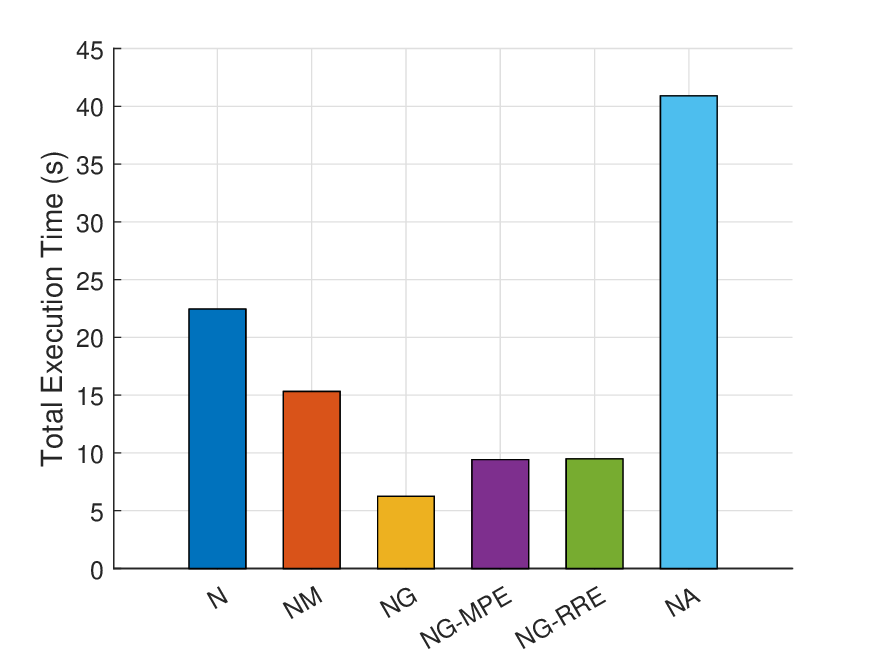}
        \caption{Execution time comparison}
        \label{fig:edinburgh_time}
    \end{subfigure}

    \caption{Performance of the methods on the \textbf{Edinburgh} network, with $\alpha=0.99$.}
    \label{fig:edinburgh_results}
\end{figure}
\begin{figure}[H]
    \centering

    \begin{subfigure}[b]{0.48\textwidth}
        \centering
        \includegraphics[width=\textwidth]{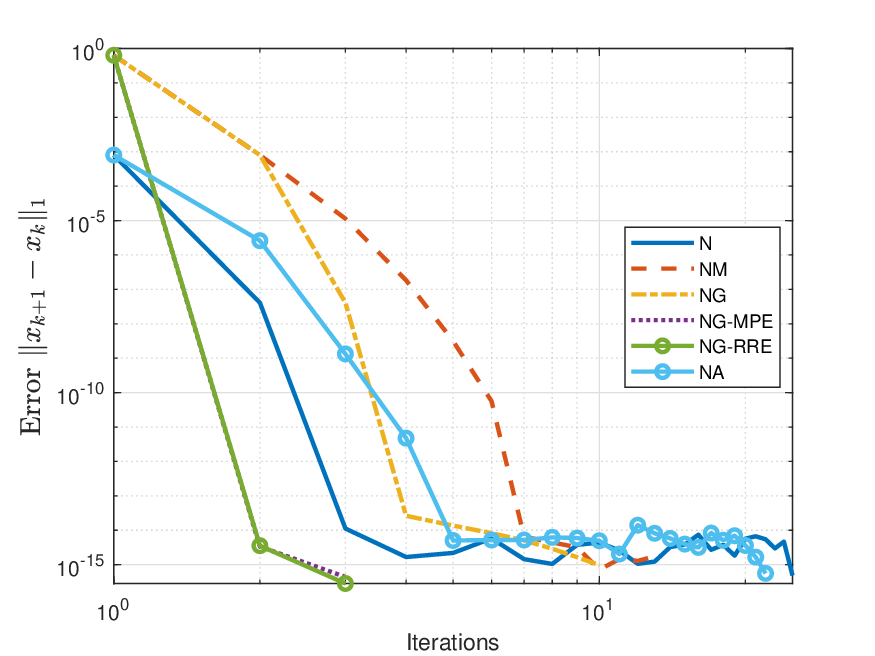}
        \caption{Convergence comparison}
        \label{fig:nottingham_residual}
    \end{subfigure}
    \hfill
    \begin{subfigure}[b]{0.48\textwidth}
        \centering
        \includegraphics[width=\textwidth]{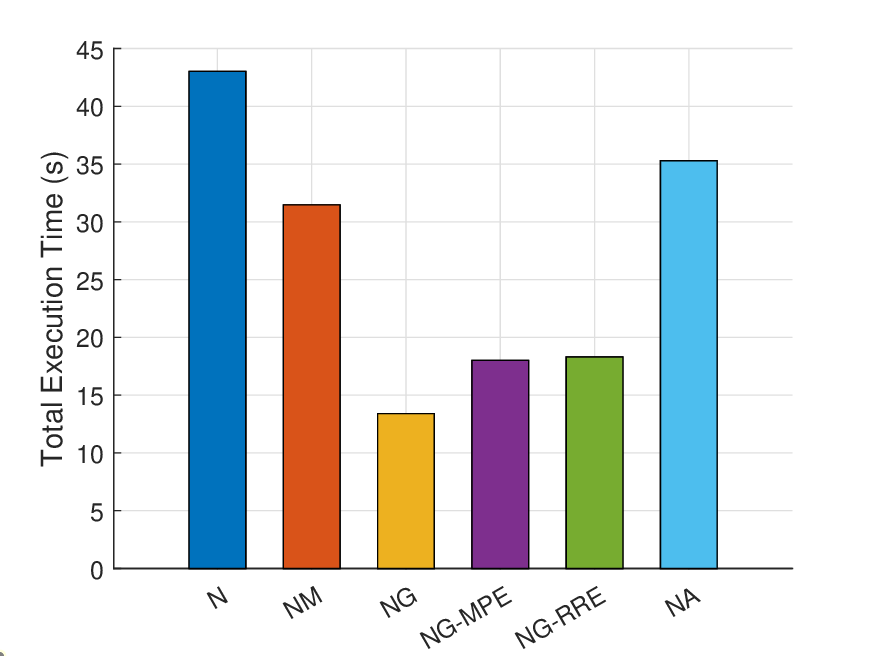}
        \caption{Execution time comparison}
        \label{fig:nottingham_time}
    \end{subfigure}

    \caption{Performance of the methods on the \textbf{Nottingham} network, with $\alpha=0.99$.}
    \label{fig:notingham_results}
\end{figure}
\begin{figure}[H]
    \centering

    \begin{subfigure}[b]{0.48\textwidth}
        \centering
        \includegraphics[width=\textwidth]{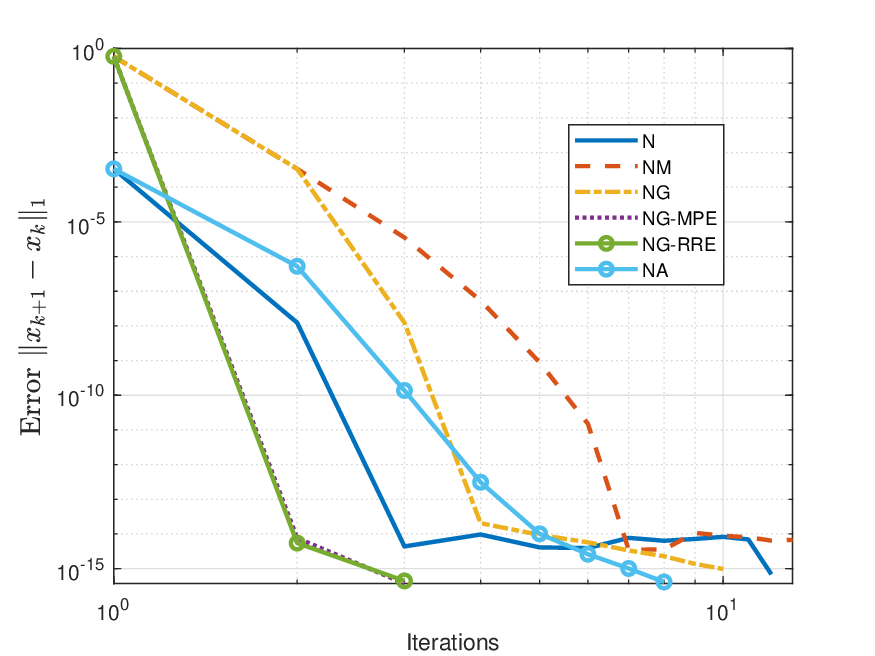}
        \caption{Convergence comparison}
        \label{fig:cardiff_residual}
    \end{subfigure}
    \hfill
    \begin{subfigure}[b]{0.48\textwidth}
        \centering
        \includegraphics[width=\textwidth]{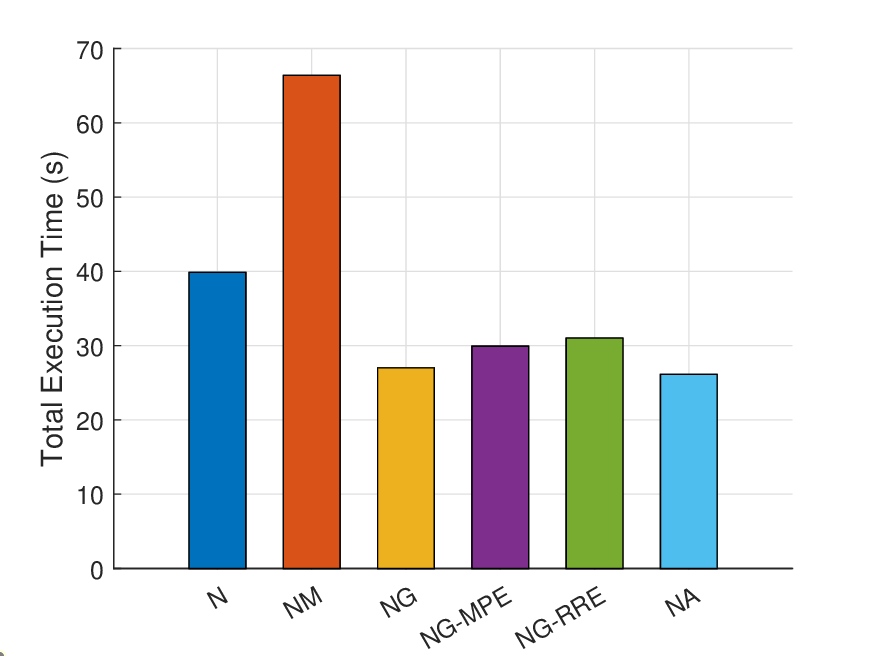}
        \caption{Execution time comparison}
        \label{fig:cardiff_time}
    \end{subfigure}

    \caption{Performance of the methods on the \textbf{Cardiff} network, with $\alpha=0.99$.}
\end{figure}
\begin{figure}[H]
\begin{subfigure}[b]{0.48\textwidth}
        \centering
        \includegraphics[width=\textwidth]{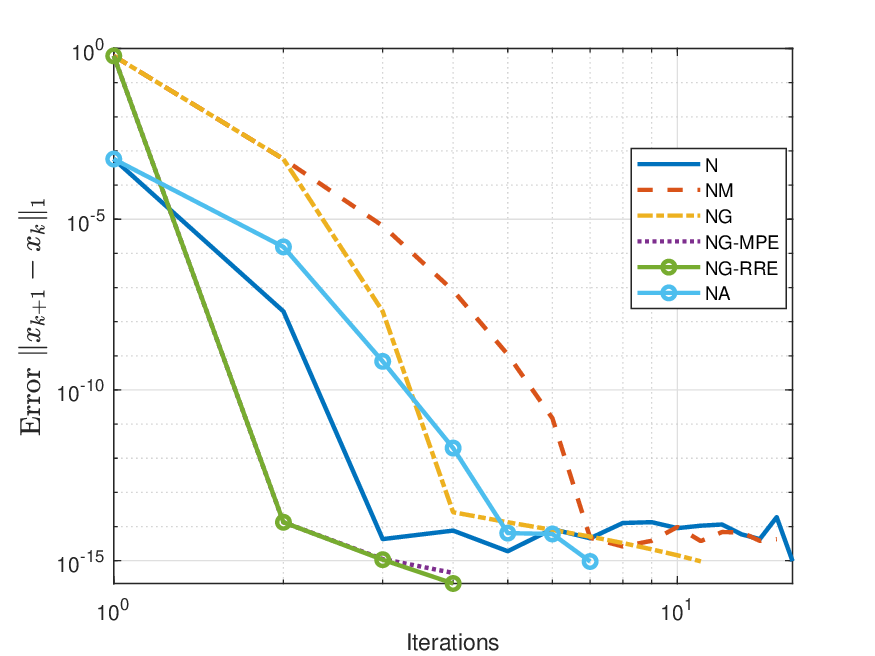}
        \caption{Convergence comparison}
        \label{fig:bristole_residual}
    \end{subfigure}
    \hfill
    \begin{subfigure}[b]{0.48\textwidth}
        \centering
        \includegraphics[width=\textwidth]{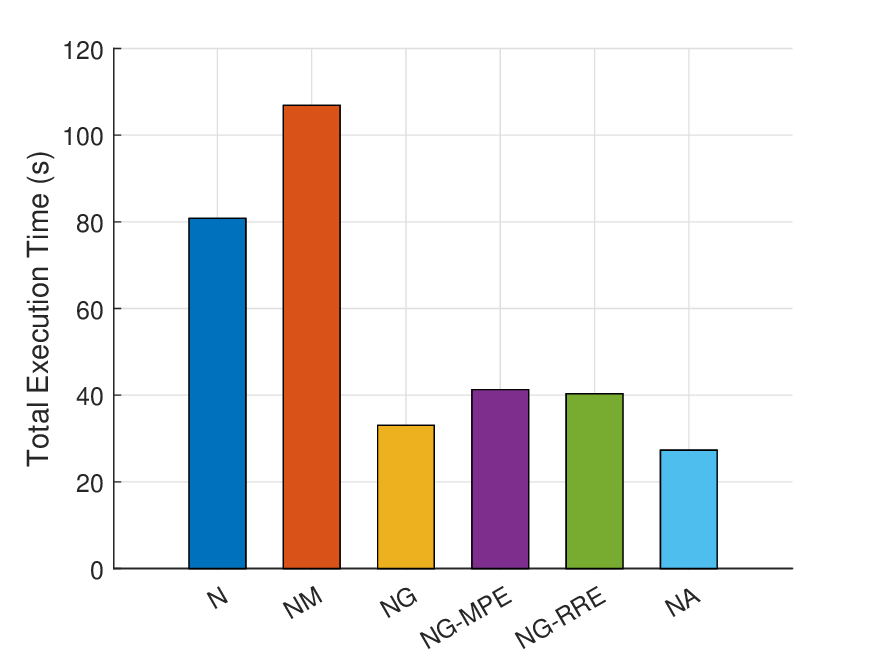}
        \caption{Execution time comparison}
        \label{fig:bristole_time}
    \end{subfigure}

    \caption{Performance of the methods on the \textbf{Bristol} network, with $\alpha=0.99$.}
    \label{fig:bristole_results}
\end{figure}
\begin{figure}[H]
\begin{subfigure}[b]{0.48\textwidth}
        \centering
        \includegraphics[width=\textwidth]{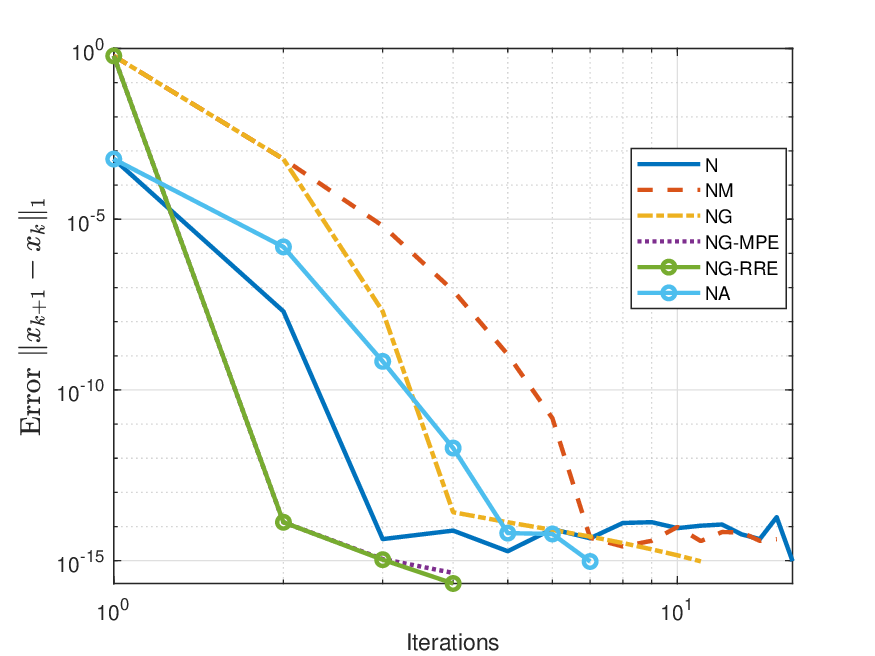}
        \caption{Convergence comparison}
        \label{fig:california_residual}
    \end{subfigure}
    \hfill
    \begin{subfigure}[b]{0.48\textwidth}
        \centering
        \includegraphics[width=\textwidth]{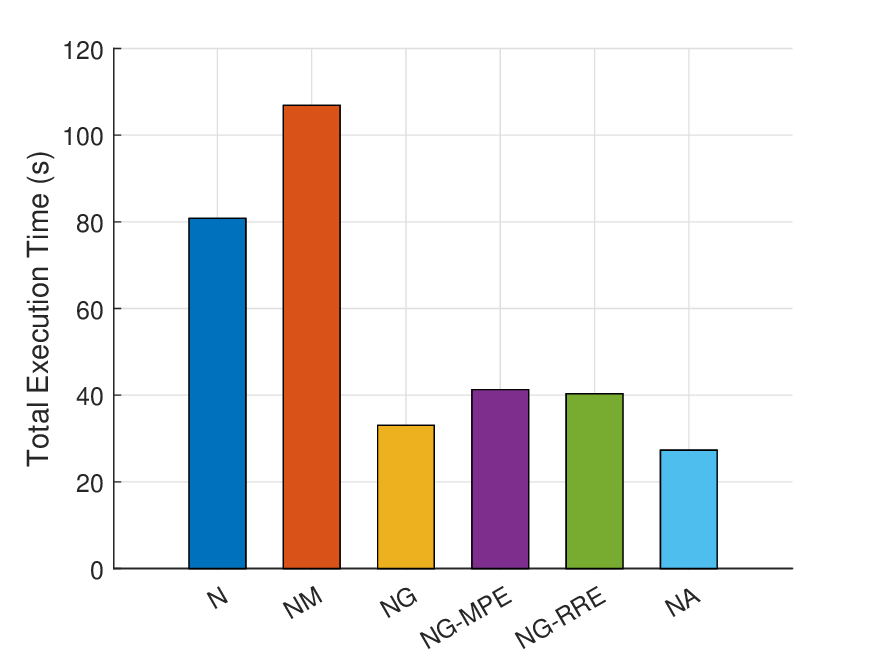}
        \caption{Execution time comparison}
        \label{fig:california_time}
    \end{subfigure}

    \caption{Performance of the methods on the \textbf{California} network, with $\alpha=0.99$.}
    \label{fig:california_results}
\end{figure}

\section{Conclusions}
In this paper, we proposed accelerated nonlinear Krylov methods to compute the multilinear PageRank vector. The Newton-GMRES scheme, combined with polynomial extrapolation (MPE, RRE) and Anderson acceleration, avoids explicit Jacobian construction and achieves faster convergence. 

Numerical experiments show that all proposed variants perform efficiently, with NG-RRE and NG-MPE offering a slight overall advantage, while NA excels in large-scale real-world problems, especially in terms of runtime. These results confirm the effectiveness and robustness of our approach for large-scale tensor computations.

\vspace{1cm}
\textbf{Author contributions}

All authors have equally contributed to the conception, writing, and revision of the manuscript.

\textbf{Financial disclosure}

The authors declare that no financial support was received for this work.

\textbf{Conflict of interest}

The authors declare that they have no conflicts of interest related to this manuscript.

\bibliographystyle{siamplain}
\bibliography{references}
\end{document}